\documentclass{amsart}

\usepackage{geometry}
\usepackage{amssymb}
\usepackage{verbatim}
\usepackage{graphicx}
\usepackage{color}

\newtheorem{prop}{Proposition}[section]
\newtheorem{teo}[prop]{Theorem}
\newtheorem*{main}{Main Theorem}
\newtheorem{lm}[prop]{Lemma}
\theoremstyle{definition}
\newtheorem{defi}[prop]{Definition}
\newtheorem{oss}[prop]{Remark}
\newtheorem*{ack}{Acknowledgments}
\newtheorem*{open}{Open problem}

\usepackage[colorlinks=true,urlcolor=blue, citecolor=red,linkcolor=blue,
linktocpage,pdfpagelabels, bookmarksnumbered,bookmarksopen]{hyperref}
\usepackage[hyperpageref]{backref}

\date{\today}

\numberwithin{equation}{section}

\title[Sharp Hardy's inequality in Sobolev-Slobodecki\u{\i} spaces]{On the sharp Hardy inequality\\ in Sobolev-Slobodecki\u{\i} spaces}

\author[Bianchi]{Francesca Bianchi}
\address[F.\ Bianchi]{Dipartimento di Scienze Matematiche, Fisiche e Informatiche
	\newline\indent
	Universit\`a di Parma
	\newline\indent
	Parco Area delle Scienze 53/a, Campus, 43124 Parma, Italy}
\email{francesca.bianchi@unipr.it}

\author[Brasco]{Lorenzo Brasco}
\address[L.\ Brasco]{Dipartimento di Matematica e Informatica
	\newline\indent
	Universit\`a degli Studi di Ferrara
	\newline\indent
	Via Machiavelli 35, 44121 Ferrara, Italy}
\email{lorenzo.brasco@unife.it}

\author[Zagati]{Anna Chiara Zagati}
\address[A.\ C.\ Zagati]{Dipartimento di Scienze Matematiche, Fisiche e Informatiche
	\newline\indent
	Universit\`a di Parma
	\newline\indent
	Parco Area delle Scienze 53/a, Campus, 43124 Parma, Italy}
\email{annachiara.zagati@unipr.it}

\begin{document}

\begin{abstract}
We study the sharp constant in the Hardy inequality for fractional Sobolev spaces defined on open subsets of the Euclidean space. We first list some properties of such a constant, as well as of the associated variational problem. 
\par
We then restrict the discussion to open {\it convex} sets and compute such a sharp constant, by constructing suitable supersolutions by means of the distance function.  
Such a method of proof works only for $s\,p\ge 1$ or for $\Omega$ being a half-space. We exhibit a simple example suggesting that this method can not work for $s\,p<1$ and $\Omega$ different from a half-space. 
\par
The case $s\,p<1$ for a generic convex set is left as an interesting open problem, except in the Hilbertian setting (i.e. for $p=2$): in this case we can compute the sharp constant in the whole range $0<s<1$. This completes a result which was left open in the literature.
\end{abstract}

\maketitle

\begin{center}
\begin{minipage}{10cm}
\small
\tableofcontents
\end{minipage}
\end{center}

\section{Introduction}

\subsection{Background} Functional inequalities of Hardy--type are certainly among the most studied ones in the theory of Sobolev spaces. The prototypical example is as follows: for an open bounded set $\Omega\subset\mathbb{R}^N$ having Lipschitz boundary and every $1<p<\infty$, there exists a constant $C=C(N,p,\Omega)>0$ such that
\[
C\,\int_\Omega \frac{|u|^p}{d_\Omega^p}\le \int_\Omega |\nabla u|^p\,dx,\qquad \mbox{ for every } u\in C^\infty_0(\Omega).
\]
Here by $d_\Omega$ we mean the distance function from the boundary, defined by
\[
d_\Omega(x)=\min_{y\in\partial\Omega} |x-y|,\qquad \mbox{ for every }x\in\Omega.
\]
We refer for example to the classical monograph \cite[Theorem 21.3]{OK} for such a result, as well as to the original paper \cite{Ne} by Jind\v{r}ich Ne\v{c}as.
\par
The boundedness and regularity assumptions on $\Omega$ are placed here just for ease of presentation: actually, the range of validity of such an inequality is much more general. 
For example, a well-known instance corresponds to the choice $\Omega=\mathbb{R}^N\setminus\{0\}$: in this case, the distance function is simply given by
\[
d_{\mathbb{R}^N\setminus\{0\}}(x)=|x|,\qquad \mbox{ for every } x\not=0,
\]
and we have the well-known inequality for $p\not=N$ (see \cite[Chapter 1, Section 1.3.1]{Maz})
\begin{equation}
	\label{Hspace}
	\left|\frac{N-p}{p}\right|^p\,\int_{\mathbb{R}^N} \frac{|u|^p}{|x|^p}\,dx\le \int_{\mathbb{R}^N} |\nabla u|^p\,dx,\qquad \mbox{ for every } u\in C^\infty_0(\mathbb{R}^N\setminus\{0\}).
\end{equation}
The constant in \eqref{Hspace} is sharp, but never attained on the natural Sobolev space attached to this inequality: this is the {\it homogeneous Sobolev space} $\mathcal{D}^{1,p}_0(\mathbb{R}^N\setminus\{0\})$, defined as the completion of $C^\infty_0(\mathbb{R}^N\setminus\{0\})$ with respect to the norm
\[
\varphi\mapsto \|\nabla \varphi\|_{L^p(\mathbb{R}^N)}.
\]
In general, it is a very interesting problem to find necessary and/or sufficient conditions assuring that an open set admits a Hardy inequality: we refer for example to \cite{An, Da94, EUS, Ha, KM, LaSo, Lewis} and \cite{Wa} for some classical results in this direction.
Even more interesting is the problem of determining the best constant in such an inequality, provided it holds true. Here we think it is mandatory to cite the two papers \cite{MMP} and \cite{MS} which contain very interesting results of general character, on the problem of determining the {\it sharp Hardy constant}. The latter is defined by 
\[
\mathfrak{h}_p(\Omega):=\inf_{u\in C^\infty_0(\Omega)} \left\{\int_\Omega |\nabla u|^p\,dx\, :\, \int_\Omega \frac{|u|^p}{d_\Omega^{p}}\,dx=1\right\},\qquad 1<p<\infty.
\] 
For a better comprehension of the contents of the present paper, it is particularly useful to recall the result of \cite[Theorem 11]{MMP}: this shows that for every $\Omega\subsetneq\mathbb{R}^N$ open {\it convex} set and every $1<p<\infty$, we have
\[
\mathfrak{h}_p(\Omega)=\mathfrak{h}_p(\mathbb{H}^N_+)=\left(\frac{p-1}{p}\right)^p,
\]
see also \cite[Theorem 1]{MaSo}.
Here we use the symbol
\begin{equation}
	\label{halfspaces}
	\mathbb{H}^1_+=(0,+\infty)\qquad \mbox{ and }\qquad \mathbb{H}^N_+=\mathbb{R}^{N-1}\times (0,+\infty),\quad \mbox{ for }N\ge 2.
\end{equation}
\subsection{Goal of the paper} In this paper, we want to tackle the same kind of problem in the setting of {\it Sobolev-Slobodecki\v{\i} fractional spaces}, that recently have attracted much attention (see \cite{DPV} for a friendly introduction to the subject). More precisely, for $1<p<\infty$ and $0<s<1$ we define the following nonlocal quantity
\[
[u]_{W^{s,p}(\mathbb{R}^N)}=\left(\iint_{\mathbb{R}^N\times\mathbb{R}^N} \frac{|u(x)-u(y)|^p}{|x-y|^{N+s\,p}}\,dx\,dy\right)^\frac{1}{p},\qquad \mbox{ for every } u\in C^\infty_0(\Omega),
\]
usually called {\it Gagliardo seminorm} or also {\it Gagliardo-Slobodecki\v{\i} seminorm}.
Then we study the sharp constant in the {\it fractional $(s,p)-$Hardy inequality}
\begin{equation}
	\label{hardystart}
	C\,\int_\Omega \frac{|u|^p}{d_\Omega^{s\,p}}\,dx\le [u]_{W^{s,p}(\mathbb{R}^N)}^p,\qquad \mbox{ for every } u\in C^\infty_0(\Omega).
\end{equation}
Such a sharp constant is obviously defined by the following variational problem
\[
\mathfrak{h}_{s,p}(\Omega):=\inf_{u\in C^\infty_0(\Omega)} \left\{[u]^p_{W^{s,p}(\mathbb{R}^N)}\, :\, \int_\Omega \frac{|u|^p}{d_\Omega^{s\,p}}\,dx=1\right\}.
\]
Observe that both integral quantities appearing in \eqref{hardystart} have the same scaling, thus it is easily seen that $\mathfrak{h}_{s,p}(\Omega)$ can not depend on the {\it size} of $\Omega$, i.e. in other words we have 
\[
\mathfrak{h}_{s,p}(\mu\,\Omega)=\mathfrak{h}_{s,p}(\Omega),\qquad \mbox{ for every } \mu>0.
\]
We first recall that the fractional counterpart of \eqref{Hspace} has been obtained in \cite[Theorem 1.1]{FSspace} by Rupert Frank and Robert Seiringer, who found the sharp constant for the punctured space $\mathbb{R}^N\setminus\{0\}$.
\par
In the case of {\it convex} sets, in \cite[Theorem 1.1]{BC} the second author and Eleonora Cinti proved that inequality \eqref{hardystart} holds for every $\Omega\subsetneq\mathbb{R}^N$, every $1<p<\infty$ and $0<s<1$. Moreover, the same result comes with an explicit lower bound on $\mathfrak{h}_{s,p}(\Omega)$: we have
\begin{equation}
	\label{lowerbaubau}
	\frac{\mathcal{C}}{s\,(1-s)}\le \mathfrak{h}_{s,p}(\Omega),
\end{equation}
for some computable constant $\mathcal{C}=\mathcal{C}(N,p)>0$. The dependence on the fractional parameter $s$ in the previous lower bound is optimal, as explained in \cite[Remark 1.2]{BC}.
The proof of \eqref{lowerbaubau} is based on the following fact: on a convex set, the power $s$ of the distance function  is {\it weakly superharmonic}, in a suitable sense, that is 
\[
(-\Delta_p)^s d_\Omega^s\ge 0,\qquad \mbox{ in }\Omega.
\]
Here $(-\Delta_p)^s$ is the {\it fractional $p-$Laplacian of order $s$}, formally defined for $1<p<\infty$ and $0<s<1$ by
\[
(-\Delta_p)^s u(x)=2\,\mathrm{P.V.} \int_{\mathbb{R}^N} \frac{|u(x)-u(y)|^{p-2}\,(u(x)-u(y))}{|x-y|^{N+s\,p}}\,dy,
\]
see Section \ref{sec:2} for the precise definition. Such an operator naturally comes into play, since its weak form is precisely the first variation of the Gagliardo-Slobodecki\v{\i} seminorm raised to the power $p$.
\vskip.2cm\noindent
The primary goal of the present paper was to improve \eqref{lowerbaubau}, by computing the sharp constant $\mathfrak{h}_{s,p}(\Omega)$ in the case of convex sets. As we will explain below, we partially succeeded in our goal. We first notice that for: 
\vskip.2cm
\begin{itemize}
	\item $p=2$, $0<s<1$ and $\Omega=\mathbb{H}^N_+$;
	\vskip.2cm
	\item $p=2$, $1/2\le s<1$ and $\Omega\subsetneq\mathbb{R}^N$ any open convex set;
\end{itemize}
\vskip.2cm
such a constant has been computed in \cite[Theorem 1.1]{BD} and \cite[Theorem 5]{FMT}, respectively. Some comments are in order on these results: actually, the statement of \cite[Theorem 1.1]{BD} is concerned with the sharp constant in the ``regional'' fractional $(s,2)-$Hardy inequality, i.\,e.
\[
C\,\int_{\mathbb{H}^N_+} \frac{|u|^2}{d_{\mathbb{H}^N_+}^{2\,s}}\,dx\le [u]^2_{W^{s,2}(\mathbb{H}^N_+)},\qquad \mbox{ for every } u\in C^\infty_0(\mathbb{H}^N_+),
\]
where for $1<p<\infty$, $0<s<1$ and every open set $E\subset\mathbb{R}^N$ we used the notation
\[
[u]_{W^{s,p}(E)}:=\left(\iint_{E\times E} \frac{|u(x)-u(y)|^p}{|x-y|^{N+s\,p}}\,dx\,dy\right)^\frac{1}{p}.
\]
Observe that this Gagliardo-Slobodecki\v{\i} seminorm on the right-hand side is computed on the smaller set $E\times E$, rather than on the whole $\mathbb{R}^N\times \mathbb{R}^N$, as in our case. 
Then the authors of \cite{BD} remark that our constant $\mathfrak{h}_{s,2}(\mathbb{H}^N_+)$ can be obtained from this ``regional'' constant, see\footnote{It should be noticed that such a formula contains a small typo: the ``regional'' term $\mathcal{C}(u,u)$ should be replaced by the ``global'' one $\mathcal{K}(u,u)$, in the notation of \cite{BD}.} \cite[formula (1.5)]{BD} or \eqref{mannaggiatte} below.
\par
Regarding \cite[Theorem 2.2]{FMT}, this has been obtained by appealing to the so-called {\it Caffarelli-Silvestre extension formula} (see \cite{CaSi}). Such a tool is quite specific of the Hilbertian setting $p=2$ and it permits to transform the {\it nonlocal} variational problem of determining $\mathfrak{h}_{s,2}$ into a weighted {\it local} variational problem, with one extra variable. However, this tool is not available for $p\not =2$ and thus a different proof is needed.
\vskip.2cm\noindent
By expanding the ideas of \cite{BD}, we will use that a crucial role in the determination of $\mathfrak{h}_{s,p}(\Omega)$ is played by positive weak (super)solutions of the equation
\begin{equation}
	\label{eqintro}
	(-\Delta_p)^s u=\lambda\,\frac{u^{p-1}}{d_\Omega^{s\,p}},\qquad \mbox{ in }\Omega.
\end{equation}
This can be regarded as the Euler-Lagrange equation of the variational problem connected with $\mathfrak{h}_{s,p}(\Omega)$.
More precisely, we will use the following fact  
\begin{equation}
	\label{charintro}
	\mathfrak{h}_{s,p}(\Omega)=\sup\Big\{\lambda\ge 0\, :\, \mbox{equation \eqref{eqintro} admits a positive local weak supersolution}\Big\},
\end{equation}
see the companion paper \cite[Theorem 1.1]{BBZ}.
Such a characterization holds true for {\it every} open set $\Omega\subsetneq\mathbb{R}^N$. Formula \eqref{charintro} permits us to refine the method of proof used in \cite{BC} and to give more precise estimates, by constructing suitable supersolutions in the case of convex sets. These will be 
\begin{equation}
	\label{optimald}
	d_\Omega^\frac{s\,p-1}{p},
\end{equation}
extended by $0$ outside $\Omega$. This method is the extension to $p\not=2$ and to general convex sets of the one employed by Krzysztof Bogdan and Bart\l omiej Dyda in the aforementioned paper \cite{BD}.

\begin{oss}[``Regional'' Hardy's inequalities]
	For ease of completeness, we recall that there exists a huge literature on the ``regional'' fractional Hardy inequality, i.e. 
	\begin{equation}
		\label{regionale}
		C\,\int_\Omega \frac{|u|^p}{d_\Omega^{s\,p}}\,dx\le [u]^p_{W^{s,p}(\Omega)},\qquad \mbox{ for every } u\in C^\infty_0(\Omega).
	\end{equation}
	It would be impossible to list all the contributions, but we like to single out \cite{Dy2}, which eventually became a landmark and contains many references on this subject. Other recent interesting papers on this problem are \cite{DK, DV, EUSV, Sk}. All these references are related to the problem of finding necessary/sufficient conditions on $\Omega$, $p$ and $s$  for \eqref{regionale} to hold.
	\par
	The problem of determining the sharp constant in \eqref{regionale}  has been addressed, apart for the above mentioned references \cite{BD, FS, FSspace}, by Michael Loss and Craig Sloane in \cite[Theorem 1.2]{LS}, still in the case of convex sets. The further restriction $s\,p>1$ is needed, since on bounded Lipschitz sets \eqref{regionale} {\it can not hold} for $s\,p\le 1$ (see \cite[Section 2]{Dy2}).
	\par
	Concerning this restriction, we conclude this remark with a problem which is still open to the best of our knowledge. For $s\,p<1$ and 
	\[
	\Omega=\Big\{(x',x_N)\in\mathbb{R}^{N-1}\times\mathbb{R}\, :\, x_N>f(x')\Big\},
	\] 
	with $f:\mathbb{R}^{N-1}\to \mathbb{R}$ a convex Lipschitz function, it is known that \eqref{regionale} holds true, see \cite[Theorem 1.1, case (T3)]{Dy2}. However {\it the sharp constant is not known in this case}, even in the case when $\Omega$ is a convex cone (not coinciding with a half-space). 
\end{oss}

\subsection{Main results: convex sets}

We need at first to settle some notation. For every $1<p<\infty$ and $0<s<1$, we define
\begin{equation}
	\label{cost1d}
	\Lambda_{s,p}:=2\,\int_0^1 \dfrac{\left|1-t^\frac{s\,p-1}{p}\right|^p}{(1-t)^{1+s\,p}}\,dt+\dfrac{2}{s\,p}.
\end{equation}
As we will see below, this will coincide with $\mathfrak{h}_{s,p}(\mathbb{H}^1_+)$. 
For every $k\in\mathbb{N}$ and $\alpha\ge 0$, we set
\[
\mathcal{I}(k;\alpha)=\int_0^{+\infty} t^k\,(1+t^2)^{-\frac{k+2+\alpha}{2}}\,dt.
\]
By indicating with $\omega_k$ the volume of the $k-$dimensional open ball of radius $1$, we define\footnote{For $N\ge 2$, such a constant could be explicitly computed in terms of the Gamma function, see for example the proof of \cite[Lemma 2.4]{FS}. For our purposes, this is not very important and we prefer not to appeal to such an explicit expression. Observe that such a constant depends on $s$ and $p$ only through their product $s\,p$.}
\begin{equation}
	\label{costd}
	C_{N,s\,p}:=\left\{\begin{array}{lr}
		(N-1)\,\omega_{N-1}\,\mathcal{I}(N-2;s\,p),& \mbox{ for } N\ge 2,\\
		&\\
		1,& \mbox{ for } N=1.
	\end{array}
	\right.
\end{equation}
We can now state our main result concerning the determination of $\mathfrak{h}_{s,p}$ for convex sets. 
\begin{main}
	Let $N\ge 1$ and $1<p<\infty$, then we have:
	\begin{enumerate} 
		\item for every $0<s<1$
		\[
		\mathfrak{h}_{s,p}(\mathbb{H}^N_+)=C_{N,s\,p}\,\Lambda_{s,p};
		\]
		\item for every $1/p\le s<1$ and every $\Omega\subsetneq\mathbb{R}^N$ open convex set 
		\[
		\mathfrak{h}_{s,p}(\Omega)=\mathfrak{h}_{s,p}(\mathbb{H}^N_+);
		\]
		\item for $p=2$, for every $0<s<1$ and every $\Omega\subsetneq\mathbb{R}^N$ open convex set 
		\[
		\mathfrak{h}_{s,2}(\Omega)=\mathfrak{h}_{s,2}(\mathbb{H}^N_+).
		\]
	\end{enumerate}
	In each case, the constant $\mathfrak{h}_{s,p}$ is not attained.
\end{main}
The previous result is obtained by combining Theorems \ref{teo:half}, \ref{teo:sp} and \ref{teo:p2} below. We first notice that point (1) could be obtained by using \cite[Theorem 1.1]{FS} for the ``regional'' case and then observing that
\begin{equation}
	\label{mannaggiatte}
	[u]^p_{W^{s,p}(\mathbb{H}^N_+)}=[u]^p_{W^{s,p}(\mathbb{R}^N)}-\frac{2\,C_{N,s\,p}}{s\,p}\,\int_{\mathbb{H}^N_+} \frac{|u|^p}{d_{\mathbb{H}^N_+}^{s\,p}}\,dx,\quad \mbox{ for every } u\in C^\infty_0(\mathbb{H}^N_+),
\end{equation}
as in \cite{BD}.
However, we prefer to give a direct proof of this fact: this is consistent with a detailed analysis of the supersolution method, taken in the present paper.
This analysis should (hopefully) lead to a better comprehension of the problem. We also think that such an analysis gives some interesting byproducts. Actually, we show in this paper that:
\begin{itemize}
	\item if $\Omega\subsetneq\mathbb{R}^N$ is an open convex set, the function $d_\Omega^\beta$ (extended by $0$ outside $\Omega$) is always a local positive weak supersolution of the equation \eqref{eqintro}
	for some $\lambda=\lambda(\beta)$, provided that the exponent $\beta$ is such that
	\[
	0\le \beta<\frac{s\,p}{p-1},
	\]
	see part (1) of Theorem \ref{teo:supersoluzioniconv};
	\vskip.2cm
	\item when $\Omega$ is an half-space, the last range can be enlarged to
	\[
	-\frac{1}{p-1}< \beta<\frac{s\,p}{p-1},
	\]
	i.e. we can even admit {\it negative} powers, see part (2) of Theorem \ref{teo:supersoluzioniconv};
	\vskip.2cm
	\item the coefficient $\lambda=\lambda(\beta)$ appearing in \eqref{eqintro} is maximal for the choice
	\[
	\beta=\frac{s\,p-1}{p},
	\]
	provided this is compatible with the previous restrictions, see Remark \ref{oss:meglio}. This neatly explains why \eqref{optimald}
	is the optimal choice, provided this is feasible. This generalizes the observation made in the proof of \cite[Theorem 1]{BD}, for $p=2$ and $\Omega=\mathbb{H}^N_+$; 
	\vskip.2cm
	\item moreover, in a very simple situation (i.e. $\Omega$ being an interval, $p=2$ and $s=1/2$), we show that the restriction $\beta\ge 0$ is unavoidable for this method to work. Indeed, we show in this case that, as soon as $\beta<0$, the function $d_\Omega^\beta$ {\it is no more a supersolution of} \eqref{eqintro} for any $\lambda> 0$. Actually, such a function is not even superharmonic in the fractional sense (see Appendix \ref{app:B}). Observe that in such a case, the function $d_\Omega^\beta$ is actually {\it convex} on $\Omega$ and non-smooth in correspondence of the maximum point of $d_\Omega$. These two facts can be held responsible for this undesired behaviour of $d_\Omega^\beta$. 
	\par
	We point out that this phenomenon does not occur in the case of half-spaces, since in this case the distance function is always smooth;
	\vskip.2cm
	\item the previous points clarify why for the range $s<1/p$ the optimal choice \eqref{optimald} should not be feasible, for convex sets not coinciding with half-spaces. In turn, this suggests that this method for constructing optimal supersolutions of \eqref{eqintro} {\it is doomed to fail} in this regime, unless $\Omega$ is a half-space;
	\vskip.2cm
	\item finally, the Hilbertian case $p=2$ is peculiar: in this case we can exploit the existence of a fractional analog of the {\it Kelvin transform} (see \cite{BZ, ROS}) to ``transplant'' the supersolution $d_{\mathbb{H}^1_+}^\beta$ from the half-line $\mathbb{H}^1_+$ to any bounded interval. This permits to determine the sharp Hardy constant for an interval, {\it without restrictions on $s$}. 
	Once we have this, we can use a ``Fubini-type argument along directions'' as in \cite{LS} (a method which goes back to \cite{Da}) and obtain the value of $\mathfrak{h}_{s,2}$, for any convex set. This in particular covers the case $0<s< 1/2$, which was left open in \cite[Theorem 5]{FMT}. Moreover, with respect to \cite{FMT}, we can obtain sharpness of the constant without any additional condition on the set, apart for its convexity.
	\vskip.2cm
\end{itemize}
Then we leave the following interesting
\begin{open}
	Prove that for $N\ge 1$, $p\in (1,\infty)\setminus\{2\}$ and $0<s<1/p$, for every $\Omega\subsetneq\mathbb{R}^N$ open convex set we have
	\[
	\mathfrak{h}_{s,p}(\Omega)=\mathfrak{h}_{s,p}(\mathbb{H}^N_+).
	\]
	Actually, in this regime, with our proof we obtain that (see Remark \ref{oss:forbice})
	\[
	C_{N,s\,p}\,\frac{2}{s\,p}\le \mathfrak{h}_{s,p}(\Omega)\le \mathfrak{h}_{s,p}(\mathbb{H}^N_+).
	\]
\end{open}

\subsection{Technical aspects of the proofs}
We want to make here some comments about our proofs. We also make some comparisons with similar results already existing in the literature. 
\par
At first, we notice that in the Main Theorem we also state that the sharp constant $\mathfrak{h}_{s,p}(\Omega)$ is not attained. Our proof of this result is based on the following general fact (see Proposition \ref{prop:nonne}): whenever an open set $\Omega\subsetneq\mathbb{R}^N$ has the following two properties:
\vskip.2cm
\begin{enumerate}
	\item[(a)] $1/d_\Omega\not\in L^1(\Omega)$;
	\vskip.2cm
	\item[(b)] there exists a local positive weak supersolution $u$ of \eqref{eqintro} with
	\[
	\lambda=\mathfrak{h}_{s,p}(\Omega)\qquad \mbox{ and }\qquad u\sim d_\Omega^\frac{s\,p-1}{p},
	\]
\end{enumerate}
then $\mathfrak{h}_{s,p}(\Omega)$ is not attained. Since property (a) is always true for a convex set, while (b) comes for free from our construction of the supersolution (recall the explanation of the previous subsection), we get immediately that $\mathfrak{h}_{s,p}(\Omega)$ is not attained for a convex set. This reasoning has been inspired to us by \cite[Theorem 1.1]{MS}.
\par
When computing the operator $(-\Delta_p)^s$ for a function of the form $t\mapsto t^\beta$ defined on the half-line (see Lemma \ref{prop:ACF}), we found it useful to use a trick taken from \cite[Appendix A]{BMS}, for tackling the singular integral in principal value sense. This trick has the advantage of giving rise to {\it homogeneous} quantities and thus convergence issues are quite easy to handle.
\par
When proving sharpness of the Hardy constant found for the half-line $\mathbb{H}^1_+$ (see the proof of Theorem \ref{teo:half}), we use some trial functions which are slightly different from those employed in \cite{BD, FS, FSspace}. Essentially, we simply approximate the ``virtual'' extremal $x^{(s\,p-1)/p}$ by adding a small $\varepsilon$ to the exponent and then letting $\varepsilon$ go to $0$. A further truncation at infinity is needed. This approach has the advantage of clarifying some Sobolev properties of power functions (see Lemma \ref{lm:potenzesobolev}), which are probably known, but not so easy to find in the literature. Moreover, it permits to treat the case $s\,p\ge 1$ and $s\,p<1$ at the same time.

\subsection{Plan of the paper}
We fix the main notation and the functional analytic setting in Section \ref{sec:2}. There we also present some inequalities and some particular trial functions that will be needed in the sequel. In Section \ref{sec:3} we prove some general properties of $\mathfrak{h}_{s,p}(\Omega)$, as well as of the variational problem naturally associated with it, for general open sets. Starting from Section \ref{sec:4}, we restrict the discussion to the case of convex sets: in this section we consider the one-dimensional case and construct suitable supersolutions of \eqref{eqintro}. This is then extended to higher dimensions in Section \ref{sec:5}. Finally, Section \ref{sec:6} contains the proof of the Main Theorem.
\par
Last but not least, Appendix \ref{app:B} contains the computation of the fractional $p-$Laplacian of a negative power of the distance, in the simple case $p=2$, $s=1/2$ and $\Omega=(0,1)\subset\mathbb{R}$.

\begin{ack}
We thank Eleonora Cinti for several useful discussions and for her kind interest towards the present work. We are also grateful to Giovanni Franzina and Simon Larson for some discussions.
\par 
L.\,B. wishes to warmly express his gratitude to Pierre Bousquet, for a clarifying discussion on the integrability of the distance function, during a staying at the Institute of Applied Mathematics and Mechanics of the University of Warsaw, in the context of the {\it Thematic Research Programme: ``Anisotropic and Inhomogeneous Phenomena''}. Iwona Chlebicka and Anna Zatorska-Goldstein are gratefully acknowledged for their kind invitation and the nice working atmosphere provided during the whole staying.
\par
F.\,B. and A.\,C.\,Z. are members of the Gruppo Nazionale per l'Analisi Matematica, la Probabilit\`a
e le loro Applicazioni (GNAMPA) of the Istituto Nazionale di Alta Matematica (INdAM).
The three authors gratefully acknowledge the financial support of the projects FAR 2019 and FAR 2020 of the University of Ferrara.
\end{ack}

\section{Preliminaries}
\label{sec:2}

\subsection{Notation}
For every $1<p<\infty$, we indicate by $J_p:\mathbb{R}\to\mathbb{R}$ the monotone increasing continuous function defined by
\[
J_p(t)=|t|^{p-2}\,t,\qquad \mbox{ for every } t\in\mathbb{R}.
\]
For $x_0\in \mathbb{R}^N$ and $R>0$, we will set
\[
B_R(x_0)=\Big\{x\in\mathbb{R}^N\, :\, |x-x_0|<R\Big\},
\]
and $\omega_N:=|B_1(x_0)|$. For an open set $\Omega\subsetneq \mathbb{R}^N$, we denote by
\[
d_\Omega(x)=\min_{y\in\partial\Omega} |x-y|,\qquad \mbox{ for every }x\in\Omega,
\]
the distance function from the boundary. 
For two open sets $\Omega'\subset\Omega \subset\mathbb{R}^N$, we will write $\Omega'\Subset\Omega$ to indicate that the closure $\overline{\Omega'}$ is a compact set contained in $\Omega$. Finally, for a measurable set $E\subset \mathbb{R}^N$, we will indicate by $1_E$ the characteristic function of such a set.

\subsection{Functional analytic setting}

For $1<p<\infty$ and $0<s<1$, we consider the fractional Sobolev space 
\[
W^{s,p}(\mathbb{R}^N)=\Big\{u\in L^p(\mathbb{R}^N)\, :\, [u]_{W^{s,p}(\mathbb{R}^N)}<+\infty\Big\},
\]
where 
\[
[u]_{W^{s,p}(\mathbb{R}^N)}:=\left(\iint_{\mathbb{R}^N\times\mathbb{R}^N} \frac{|u(x)-u(y)|^p}{|x-y|^{N+s\,p}}\,dx\,dy\right)^\frac{1}{p}.
\]
This is a reflexive Banach space, when endowed with the natural norm
\[
\|u\|_{W^{s,p}(\mathbb{R}^N)}=\|u\|_{L^p(\mathbb{R}^N)}+[u]_{W^{s,p}(\mathbb{R}^N)}.
\]
For an open set $\Omega\subset\mathbb{R}^N$, we indicate by $\widetilde{W}^{s,p}_0(\Omega)$ the closure of $C^\infty_0(\Omega)$ in $W^{s,p}(\mathbb{R}^N)$. By the Hahn-Banach Theorem, this is a weakly closed subspace of $W^{s,p}(\mathbb{R}^N)$, as well.
\par
Occasionally, for an open set $\Omega\subset\mathbb{R}^N$, we will need the fractional Sobolev space defined by
\[
W^{s,p}(\Omega)=\Big\{u\in L^p(\Omega)\, :\, [u]_{W^{s,p}(\Omega)}<+\infty\Big\},
\]
where 
\[
[u]_{W^{s,p}(\Omega)}:=\left(\iint_{\Omega \times \Omega} \frac{|u(x)-u(y)|^p}{|x-y|^{N+s\,p}}\,dx\,dy\right)^\frac{1}{p}.
\]
We will denote by $W^{s,p}_{\rm loc}(\Omega)$ the space of functions $u\in L^p_{\rm loc}(\Omega)$ such that $u\in W^{s,p}(\Omega')$ for every $\Omega'\Subset \Omega$.
\par
For $0<\beta<\infty$, we also denote by $L^\beta_{s\,p}(\mathbb{R}^N)$ the following weighted Lebesgue space
\[
L^\beta_{s\,p}(\mathbb{R}^N)=\left\{u\in L^\beta_{\rm loc}(\mathbb{R}^N)\, :\, \int_{\mathbb{R}^N} \frac{|u(x)|^\beta}{(1+|x|)^{N+s\,p}}\,dx<+\infty\right\}.
\]
We observe that this is a Banach space for $\beta\ge 1$, when endowed with the natural norm. 
\par
For $1<p<\infty$ and $0<s<1$, in a open set $\Omega\subsetneq\mathbb{R}^N$ we want to consider the equation
\begin{equation}
\label{equazione}
(-\Delta_p)^s u=\lambda\,\frac{|u|^{p-2}\,u}{d_\Omega^{s\,p}},\qquad \mbox{ in }\Omega,
\end{equation}
where $\lambda\ge 0$. The symbol $(-\Delta_p)^s$ stands for the {\it fractional $p-$Laplacian of order $s$}, defined in weak form by the first variation of the convex functional
\[
u\mapsto \frac{1}{p}\,[u]_{W^{s,p}(\mathbb{R}^N)}^p.
\]
\begin{defi}
\label{defi:subsuper}
We say that $u\in W^{s,p}_{\rm loc}(\Omega)\cap L^{p-1}_{s\,p}(\mathbb{R}^N)$ is a 
\begin{itemize}
\item {\it local weak supersolution} of \eqref{equazione} if 
\[
\iint_{\mathbb{R}^N\times \mathbb{R}^N} \frac{J_p(u(x)-u(y))\,(\varphi(x)-\varphi(y))}{|x-y|^{N+s\,p}}\,dx\,dy\ge \lambda\,\int_\Omega \frac{|u(x)|^{p-2}\,u(x)}{d_\Omega(x)^{s\,p}}\,\varphi(x)\,dx,
\]
for every non-negative $\varphi\in W^{s,p}(\mathbb{R}^N)$ with compact support in $\Omega$;
\vskip.2cm
\item {\it local weak subsolution} of \eqref{equazione} if 
\[
\iint_{\mathbb{R}^N\times \mathbb{R}^N} \frac{J_p(u(x)-u(y))\,(\varphi(x)-\varphi(y))}{|x-y|^{N+s\,p}}\,dx\,dy\le \lambda\,\int_\Omega \frac{|u(x)|^{p-2}\,u(x)}{d_\Omega(x)^{s\,p}}\,\varphi(x)\,dx,
\]
for every non-negative $\varphi\in W^{s,p}(\mathbb{R}^N)$ with compact support in $\Omega$;
\vskip.2cm
\item {\it local weak solution} of \eqref{equazione} if it is both a local weak supersolution and a local weak subsolution.
\end{itemize}
\end{defi}
\begin{oss}
It is not difficult to see that under the assumptions taken on $u$ and the test function $\varphi$, the previous definition is well-posed, i.e.
\[
 \frac{J_p(u(x)-u(y))\,(\varphi(x)-\varphi(y))}{|x-y|^{N+s\,p}}\in L^1(\mathbb{R}^N\times\mathbb{R}^N).
\]
\end{oss}
The following simple result will be useful: its proof is simply based on standard properties of convolutions and it is thus omitted (see for example \cite[Lemma 11]{FSV}).
This shows in particular that in Definition \ref{defi:subsuper} we can simply take $\varphi\in C^\infty_0(\Omega)$, considered to be $0$ on the complement $\mathbb{R}^N\setminus\Omega$.
\begin{lm}
\label{lm:libeccio}
Let $1<p<\infty$ and $0<s<1$. Let $\Omega\subseteq\mathbb{R}^N$ be an open set. If $\varphi\in W^{s,p}(\mathbb{R}^N)$ has compact support in $\Omega$, then we have $\varphi\in \widetilde{W}^{s,p}_0(\Omega)$.
\end{lm}

\subsection{Pointwise inequalities}
We recall the following discrete version of Picone's inequality, taken from \cite[Proposition 4.2]{BF} (see also \cite[Lemma 2.6]{FSspace}). We explicitly state the equality cases.
\begin{lm}[Discrete Picone's inequality]
\label{lm:picone}
Let $1<p<\infty$, for every $a,b>0$ and $c,d\ge 0$ we have 
\[
J_p(a-b)\,\left(\frac{c^p}{a^{p-1}}-\frac{d^p}{b^{p-1}}\right)\le |c-d|^p.
\]
Moreover, equality holds if and only if 
\[
\frac{c}{a}=\frac{d}{b}.
\]
\end{lm}
\begin{proof}
We first observe that if $c=0$, the inequality is equivalent to 
\[
J_p(b-a)\,\frac{d^p}{b^{p-1}}\le d^p.
\]
If we also have $d=0$, then this is trivially true. If $d>0$, then this is equivalent to 
\[
J_p(b-a)\,\frac{1}{b^{p-1}}\le 1 \qquad\mbox{ that is }\qquad \left|1-\frac{a}{b}\right|^{p-2}\,\left(1-\frac{a}{b}\right)\le 1.
\]
Since $a$ and $b$ are both positive, it is easily seen that the last inequality is true, actually with the strict inequality sign.
\par
We then suppose $c\not= 0$: we first observe that the left-hand side can be rewritten as
\[
\begin{split}
J_p(a-b)\,\left(\frac{c^p}{a^{p-1}}-\frac{d^p}{b^{p-1}}\right)&=J_p\left(1-\frac{b}{a}\right)\,\left(c^p-d^p\,\left(\frac{a}{b}\right)^{p-1}\right)\\
&=c^p\,J_p\left(1-\frac{b}{a}\right)\,\left(1-\left(\frac{d}{c}\right)^p\,\left(\frac{a}{b}\right)^{p-1}\right),
\end{split}
\]
thanks to the homogeneity of $J_p$.
If we introduce the shortcut notation
\[
t=\frac{b}{a}\qquad \mbox{ and }\qquad A=\frac{d}{c},
\]
we then get that the claimed inequality is equivalent to 
\[
J_p(1-t)\,\left(1-\frac{A^p}{t^{p-1}}\right)\le |1-A|^p,\qquad \mbox{ for every } t>0, \ A\ge 0.
\]
It is not difficult to see that the function
\[
\Phi(t)=J_p(1-t)\,\left(1-\frac{A^p}{t^{p-1}}\right),
\]
is monotone increasing for $t<A$ and monotone decreasing for $t>A$. The choice $t=A$ thus corresponds to the {\it unique} maximum point, for which we have
\[
\Phi(t)\le \Phi(A)=|1-A|^p.
\]
This concludes the proof.
\end{proof}

The following simple inequality will be useful somewhere in Section \ref{sec:4}.
\begin{lm}
\label{lm:lemmino}
Let $\beta\not =0$, then we have 
\[
|1-\tau^\beta|\le |\beta|\,\max\big\{\tau^{\beta-1},1\big\}\,(1-\tau),\qquad \mbox{ for every } \tau\in(0,1),
\]
and
\[
|1-\tau^\beta|\le |\beta|\,\max\left\{\tau^\beta,\frac{1}{\tau}\right\}\,(1-\tau),\qquad \mbox{ for every } \tau>1.
\]
\end{lm}
\begin{proof}
We start with the case $\beta>0$. By basic Calculus, we have for $\tau\in(0,1)$
\[
|1-\tau^\beta|=1-\tau^\beta=\beta\,\xi^{\beta-1}\,(1-\tau),
\]
for some $\tau<\xi<1$. We observe that the quantity $\xi\mapsto\beta\,\xi^{\beta-1}$ is increasing for $\beta>1$ and decreasing for $0<\beta<1$. This gives the desired conclusion.
\par
The case $\beta<0$ is treated similarly. We have this time for $\tau\in(0,1)$
\[
|1-\tau^\beta|=\tau^\beta-1=\beta\,\xi^{\beta-1}\,(\tau-1)=(-\beta)\,\xi^{\beta-1}\,(1-\tau),
\]
for some $\tau<\xi<1$. By using that $\xi^{\beta-1}<\tau^{\beta-1}$, we get the conclusion, here as well.
\par
Finally, for $\tau>1$ it is sufficient to write 
\[
|1-\tau^\beta|=\left|1-\left(\frac{1}{\tau}\right)^{-\beta}\right|,
\]
and then use the previous inequality, with $-\beta$ in place of $\beta$ and $1/\tau$ in place of $\tau$.
\end{proof}

\subsection{Some trial functions}
The next result is a sort of interpolation--type inequality, for smooth functions. It is useful in order to prove some Leibniz--type formulas in fractional Sobolev spaces.
\begin{lm}
\label{lm:besov}
Let $1<p<\infty$ and $0<s<1$, then for every $\varphi\in C^1_0(\mathbb{R}^N)$ we have
\[
\sup_{x\in\mathbb{R}^N}\int_{\mathbb{R}^N} \frac{|\varphi(x)-\varphi(y)|^p}{|x-y|^{N+s\,p}}\,dy\le \frac{C}{s\,(1-s)}\,\|\nabla \varphi\|^{s\,p}_{L^\infty(\mathbb{R}^N)}\,\|\varphi\|^{(1-s)\,p}_{L^\infty(\mathbb{R}^N)},
\]
for some $C=C(N,p)>0$.
\end{lm}
\begin{proof}
We pick $\delta>0$, then for $x\in\mathbb{R}^N$ we split the integral in two parts
\[
\begin{split}
\int_{\mathbb{R}^N} \frac{|\varphi(x)-\varphi(y)|^p}{|x-y|^{N+s\,p}}\,dy&=\int_{B_\delta(x)} \frac{|\varphi(x)-\varphi(y)|^p}{|x-y|^{N+s\,p}}\,dy+\int_{\mathbb{R}^N\setminus B_\delta(x)} \frac{|\varphi(x)-\varphi(y)|^p}{|x-y|^{N+s\,p}}\,dy\\
&\le \|\nabla \varphi\|^p_{L^\infty(\mathbb{R}^N)}\,\int_{B_\delta(x)} |x-y|^{p\,(1-s)-N}\,dy\\
&+ 2^p\,\|\varphi\|^p_{L^\infty(\mathbb{R}^N)} \,\int_{\mathbb{R}^N\setminus B_\delta(x)} |x-y|^{-N-s\,p}\,dy\\
&=\frac{N\,\omega_N}{p}\, \left(\frac{\|\nabla \varphi\|^p_{L^\infty(\mathbb{R}^N)}}{1-s}\,\delta^{p\,(1-s)}+2^p\,\frac{\|\varphi\|^p_{L^\infty(\mathbb{R}^N)}}{s}\,\delta^{-s\,p}\right).
\end{split}
\]
By optimizing in $\delta>0$, we get the desired result.
\end{proof}

\begin{lm}
\label{lm:ammissibili1}
Let $1<p<\infty$ and $0<s<1$. Let $\Omega\subseteq\mathbb{R}^N$ be an open set, then for every $\eta\in C^1_0(\Omega)$ and $u\in W^{s,p}_{\rm loc}(\Omega)$, the function $\eta\,u$ is compactly supported in $\Omega$ and belongs to $W^{s,p}(\mathbb{R}^N)$. In particular, we have
\[
\eta\,u\in \widetilde{W}^{s,p}_0(\Omega).
\]
\end{lm}
\begin{proof}
We consider both $\eta$ and $u$ to be extended by $0$ to $\mathbb{R}^N\setminus\Omega$. In light of Lemma \ref{lm:libeccio}, we only need to show that $\eta\,u\in W^{s,p}(\mathbb{R}^N)$. We take $\Omega''\Subset \Omega'\Subset\Omega$ such that the support of $\eta$ is contained in $\Omega''$. Then we may write 
\[
\begin{split}
[\eta\,u]^p_{W^{s,p}(\mathbb{R}^N)}&=[\eta\,u]_{W^{s,p}(\Omega')}^p+2\,\iint_{\Omega'\times (\mathbb{R}^N\setminus{\Omega'})} \frac{|\eta\,(x)\,u(x)|^p}{|x-y|^{N+s\,p}}\,dx\,dy\\
&=[\eta\,u]_{W^{s,p}(\Omega')}^p+2\,\iint_{\Omega''\times (\mathbb{R}^N\setminus{\Omega'})} \frac{|\eta\,(x)\,u(x)|^p}{|x-y|^{N+s\,p}}\,dx\,dy,
\end{split}
\]
where we used that $\eta$ vanishes outside $\Omega''$. For the first term, by using Minkowksi's inequality and Lemma \ref{lm:besov}, we can estimate it from above by means of the following Leibniz--type rule
\[
\begin{split}
[\eta\,u]_{W^{s,p}(\Omega')}&\le \left(\int_{\Omega'} |u(x)|^p\,\left(\int_{\Omega'}\frac{|\eta(x)-\eta(y)|^p}{|x-y|^{N+s\,p}}\,dy \right)\,dx\right)^\frac{1}{p}\\
&+\left(\int_{\Omega'} |\eta(y)|^p\,\left(\int_{\Omega'}\frac{|u(x)-u(y)|^p}{|x-y|^{N+s\,p}}\,dx \right)\,dy\right)^\frac{1}{p}\\
& \le \left(\frac{C}{s\,(1-s)}\right)^\frac{1}{p}\,\|u\|_{L^p(\Omega')}\,\|\nabla \eta\|^{s}_{L^\infty(\mathbb{R}^N)}\,\|\eta\|^{1-s}_{L^\infty(\mathbb{R}^N)}+\|\eta\|_{L^\infty(\mathbb{R}^N)}\, [u]_{W^{s,p}(\Omega')}<+\infty.
\end{split}
\]
For the second term, we have 
\[
\begin{split}
2\,\iint_{\Omega''\times (\mathbb{R}^N\setminus{\Omega'})} \frac{|\eta\,(x)\,u(x)|^p}{|x-y|^{N+s\,p}}\,dx\,dy&\le 2\,\|\eta\|^p_{L^\infty(\mathbb{R}^N)}\,\int_{\Omega''}|u(x)|^p\left(\int_{\mathbb{R}^N\setminus{\Omega'}} \frac{dy}{|x-y|^{N+s\,p}}\right)\,dx\\
&\le
2\,\|\eta\|^p_{L^\infty(\mathbb{R}^N)}\,\int_{\Omega''}|u(x)|^p\left(\int_{\mathbb{R}^N\setminus B_\mathfrak{d}(x)} \frac{dy}{|x-y|^{N+s\,p}}\right)\,dx\\
&=\frac{2\,N\,\omega_N}{s\,p}\,\frac{\|\eta\|^p_{L^\infty(\mathbb{R}^N)}}{\mathfrak{d}^{s\,p}}\,\int_{\Omega''}|u(x)|^p\,dx<+\infty,
\end{split}
\]
where we set $\mathfrak{d}=\mathrm{dist}(\Omega'',\partial\Omega')>0$.
\end{proof}
The following technical result will be used in order to verify the sharpness of our Hardy's inequality on the half-line $\mathbb{H}^1_+=(0,+\infty)$.
\begin{lm}\label{lm:lemmaproduct}
Let $1<p<\infty$ and $0<s<1$. For $M>0$, we take $u \in W^{s,p}((0,M))$ and extend it by $0$ outside $(0,M)$. We also suppose that there exist $C>0$ and $\beta>(s\,p-1)/p$ such that
\[
|u(x)|\le C\,x^\beta,\qquad \mbox{ for a.\,e. } x\in(0,M). 
\]
Then for every $\eta \in C^{\infty}_0((-\infty,M))$, we have
	\[
	u\,\eta \in \widetilde{W}^{s,p}_0(\mathbb{H}^1_+).
	\]
	Moreover, the following estimates hold
	\begin{equation}
	\label{eq:stimaseminorma1}
	\begin{split}
	[ u \, \eta]^p_{W^{s,p}(\mathbb{R})}& \le [u\,\eta]^p_{W^{s,p}((0,M))}+\frac{2}{s\,p}\,\int_{0}^{M} \frac{|u(x)\,\eta(x)|^p}{|x|^{s\,p}} \, dx+\frac{2\,M^{p-s\,p}}{s\,p}\,\|\eta'\|^p_{L^\infty(\mathbb{R})}\,\|u\|^p_{L^p((0,M))},
	\end{split}
	\end{equation}
	and
	\begin{equation}
	\label{eq:stimaseminorma2}
[u\,\eta]_{W^{s,p}((0,M))}\le \|\eta\|_{L^\infty(\mathbb{R})}\,[u]_{W^{s,p}((0,M))}+\left(\frac{C}{s\,(1-s)}\right)^\frac{1}{p}\,\|u\|_{L^p((0,M))}\,\|\eta'\|^s_{L^\infty(\mathbb{R})}\,\|\eta\|^{1-s}_{L^\infty(\mathbb{R})},
\end{equation}
	for some $C=C(p)>0$.
\end{lm}

\begin{proof}
We start by observing that $\widetilde{W}^{s,p}_0(\mathbb{H}^1_+)$ can be identified with the space of functions in $W^{s,p}(\mathbb{R})$ which vanish almost everywhere in $(-\infty,0]$, thanks to \cite[Theorem 6]{FSV}. By construction, it is then sufficient to prove that $\eta\,u\in W^{s,p}(\mathbb{R})$.
\par
It is easy to see that $u \, \eta \in L^p(\mathbb{R})$, hence let us focus on proving that $u \, \eta$ has a finite $W^{s,p}$ seminorm. 
By construction, this function vanishes almost everywhere outside $(0,M)$. We decompose the seminorm as follows
	\[ 
	\begin{split}
	[ u \, \eta]^p_{W^{s,p}(\mathbb{R})} 
	& = [ u \, \eta]^p_{W^{s,p}((0,M))} + 2\,\int_0^{M} \int_{-\infty}^0 \frac{|u(x) \, \eta(x)|^p}{|x-y|^{1+s\,p}} \, dy \, dx +2\,\int_0^{M} \int_M^{+\infty} \frac{|u(x) \, \eta(x)|^p}{|x-y|^{1+s\,p}} \, dy \, dx\\
	&= [ u \, \eta]^p_{W^{s,p}((0,M))} + \frac{2}{s\,p}\,\int_0^{M} \frac{|u(x) \, \eta(x)|^p}{|x|^{s\,p}} \, dx+\frac{2}{s\,p}\,\int_0^{M} \frac{|u(x) \, \eta(x)|^p}{(M-x)^{s\,p}}  \, dx.
	\end{split}
	\]	
In order to estimate the first term on the right-hand side, we proceed similarly as in the proof of Lemma \ref{lm:ammissibili1}, so to get
	\[ 
	\begin{split}
	[ u \, \eta]_{W^{s,p}((0,M))} &\le \left[ \int_{0}^{M} |u(x)|^p \left(\int_{0}^{M} \frac{|\eta(x)-\eta(y)|^p}{|x-y|^{1+s\,p}} \, dy\right) \, dx \right]^\frac{1}{p} \\
	&+ \left[ \int_{0}^{M} |\eta(y)|^p \left(\int_{0}^{M} \frac{|u(x)-u(y)|^p}{|x-y|^{1+s\,p}} \, dx\right) \, dy \right]^\frac{1}{p} \\
	&\le \left(\frac{C}{s\,(1-s)} \right)^{\frac{1}{p}}\, \|u\|_{L^p((0,M))} \,\|\eta'\|^{s}_{L^\infty(\mathbb{R})}\,\|\eta\|^{1-s}_{L^\infty(\mathbb{R})}+ \| \eta \|_{L^{\infty}((0,M))} \, [u]_{W^{s,p}((0,M))}. 
	\end{split} \]
In the last inequality, we applied again Lemma \ref{lm:besov}. As for the other terms, we observe that
	\[
	\begin{split}
	\int_0^{M} \frac{|u(x) \, \eta(x)|^p}{|x|^{s\,p}} \, dx,
	\end{split}
	\]
	is finite, thanks to the growth assumption on $u$. Finally, by using that $\eta\in C^\infty_0((-\infty,M))$, we have 
	\[
	|u(x) \, \eta(x)|^p=|u(x)|^p\,|\eta(x)-\eta(M)|^p\le \|\eta'\|_{L^\infty(\mathbb{R})}^p\, |u(x)|^p\,(M-x)^p,\qquad \mbox{ for a.\,e. } x\in(0,M),
	\]
so that we can infer
\[
\int_0^{M} \frac{|u(x) \, \eta(x)|^p}{(M-x)^{s\,p}}  \, dx\le M^{p-s\,p}\, \|\eta'\|^p_{L^\infty(\mathbb{R})}\,\int_0^M |u(x)|^p\,dx<+\infty.
\]
This completes the proof.
\end{proof}

\begin{lm}[Fractional hidden convexity]
\label{lm:FHC}
Let $1<p<\infty$ and $0<s<1$. Let $\Omega\subseteq\mathbb{R}^N$ be an open set, for every two non-negative functions $u,v\in \widetilde{W}^{s,p}_0(\Omega)$, we set
\[
\sigma=\left(\frac{1}{2}\,u^p+\frac{1}{2}\,v^p\right)^\frac{1}{p}.
\]
Then $\sigma\in \widetilde{W}^{s,p}_0(\Omega)$ and there holds
\begin{equation}
\label{eq:hiddenconvexity}
[\sigma]^p_{W^{s,p}(\mathbb{R}^N)} \le \frac{1}{2}\, [u]^p_{W^{s,p}(\mathbb{R}^N)} + \frac{1}{2}\, [v]^p_{W^{s,p}(\mathbb{R}^N)}.
\end{equation}
Moreover, if equality holds in \eqref{eq:hiddenconvexity}, then there exists a constant $C$ such that 
\[ 	
u=C\,v,\qquad \text{ a.\,e. in } \Omega.
\]  
\end{lm}
\begin{proof}
	The proof of \eqref{eq:hiddenconvexity} and the identification of equality cases are contained in \cite[Lemma 4.1 \& Theorem 4.2]{FP}. We just show here that $\sigma$ belongs to the relevant fractional Sobolev space, a fact that seems to have been overlooked in the literature. We first notice that 
	\[ 
	\int_{\Omega} |\sigma|^{p} \, dx = \frac{1}{2} \int_{\Omega} u^{p} \, dx + \frac{1}{2} \int_{\Omega} v^{p} \, dx < +\infty,
	\]
	and by \eqref{eq:hiddenconvexity} we have in particular
	\[
	[\sigma]_{W^{s,p}(\mathbb{R}^N)}<+\infty.
	\]
	This shows that $\sigma\in W^{s,p}(\mathbb{R}^N)$.
	We now consider $\{u_n\}_{n \in \mathbb{N}}, \{v_n\}_{n \in \mathbb{N}} \subset C^{\infty}_0(\Omega)$ two sequences which converge respectively to $u$ and $v$ in $W^{s,p}(\mathbb{R}^N)$. Since $u$ and $v$ are positive, without loss of generality we can take $u_n$ and $v_n$ to be non-negative. Moreover, up to pass to a subsequence, we can suppose to have almost everywhere convergence.
	\par
	We set
	\[ 
	\sigma_n = \left(\frac{1}{2}\,\left(u_n+\frac{1}{n}\right)^{p} + \frac{1}{2}\, \left(v_n+\frac{1}{n}\right)^p \right)^\frac{1}{p}-\frac{1}{n}, \qquad \mbox{ for every } n \in \mathbb{N}\setminus\{0\},
	\]
	and observe that $\{\sigma_n\}_{n\in\mathbb{N}}\subset C^\infty_0(\Omega)\subset \widetilde{W}^{s,p}_0(\Omega)$. Moreover, $\sigma_n$ converges to $\sigma$ almost everywhere, as $n$ goes to $\infty$. We claim that 
	\begin{equation}
	\label{BL}
	\lim_{n\to\infty} \|\sigma_n- \sigma\|_{L^p(\Omega)}=0.
	\end{equation}
	Indeed, thanks to Fatou's Lemma, it holds that
	\[ 
	\liminf_{n \to \infty} \int_{\Omega} |\sigma_n|^p \, dx \ge \int_{\Omega} |\sigma|^p \, dx.
	\]
	Conversely, we observe that\footnote{This follows by noticing that the function
	\[ h(\varepsilon)=\left(\frac{1}{2}\,\left(a+\varepsilon\right)^{p} + \frac{1}{2}\, \left(b+\varepsilon\right)^p \right)^\frac{1}{p}-\varepsilon, \quad \mbox{ for every } a, b\ge 0,
	\]
is monotone decreasing with respect to $\varepsilon\ge 0$, thus $h(\varepsilon)\le h(0)$.}
	\[
	\sigma_n\le \left(\frac{1}{2}\,u_n^{p} + \frac{1}{2}\, v_n^p \right)^\frac{1}{p},\qquad \mbox{ for every } n\in\mathbb{N}\setminus\{0\}.
	\]
By raising to the power $p$ and taking the limit, we get	
	\[ 
	\limsup_{n \to \infty} \int_{\Omega} |\sigma_n|^p \, dx \le \limsup_{n \to \infty}\left[ \frac{1}{2}\, \int_{\Omega}u_n^{p}\,dx + \frac{1}{2}\,\int_\Omega v_n^p \, dx\right] = \int_{\Omega} |\sigma|^p \, dx.
	\]
These facts entail that we have convergence of the $L^p$ norms. By joining this with the almost everywhere convergence, we get \eqref{BL} from the so-called {\it Br\'ezis-Lieb Lemma} (see \cite[Theorem 1]{BL}).
\par
	We also observe that $[\sigma_n]_{W^{s,p}(\mathbb{R}^N)}$ is bounded. Indeed, we can apply the convexity inequality \eqref{eq:hiddenconvexity} as follows
	\[
	\begin{split}
	[\sigma_n]^p_{W^{s,p}(\mathbb{R}^N)}=\left[\sigma_n+\frac{1}{n}\right]^p_{W^{s,p}(\mathbb{R}^N)}&\le \frac{1}{2}\,\left[u_n+\frac{1}{n}\right]_{W^{s,p}(\mathbb{R}^N)}^p+\frac{1}{2}\,\left[v_n+\frac{1}{n}\right]_{W^{s,p}(\mathbb{R}^N)}^p\\
	&=\frac{1}{2}\,[u_n]_{W^{s,p}(\mathbb{R}^N)}^p+\frac{1}{2}\,[v_n]_{W^{s,p}(\mathbb{R}^N)}^p,
	\end{split}
	\]
	and observe that the last terms are uniformly bounded, by construction. The uniform bound on $\|\sigma_n\|_{W^{s,p}(\mathbb{R}^N)}$ and the reflexivity of the space $W^{s,p}(\mathbb{R}^N)$ entail that $\sigma_n$ weakly converges, up to subsequences, to a function in $\widetilde{W}^{s,p}_0(\Omega)$, the latter being a weakly closed subspace of $W^{s,p}(\mathbb{R}^N)$. By the uniqueness of the limit, such a function must coincide with $\sigma$, which then belongs to $\widetilde{W}^{s,p}_0(\Omega)$.
\end{proof}

\section{The sharp fractional Hardy constant}
\label{sec:3}
Let $1<p<\infty$ and $0<s<1$. For an open set $\Omega\subsetneq\mathbb{R}^N$, we define its {\it sharp fractional $(s,p)-$Hardy constant}, i.e.
\[
\mathfrak{h}_{s,p}(\Omega):=\inf_{u\in C^\infty_0(\Omega)} \left\{[u]^p_{W^{s,p}(\mathbb{R}^N)}\, :\, \int_\Omega \frac{|u|^p}{d_\Omega^{s\,p}}\,dx=1\right\}.
\]
It is not difficult to see that 
\begin{equation}
\label{scalina}
\mathfrak{h}_{s,p}(\Omega)=\mathfrak{h}_{s,p}(t\,\Omega+x_0),\qquad \mbox{ for every } t>0,\ x_0\in\Omega.
\end{equation}
\begin{oss}
By definition of the space $\widetilde{W}^{s,p}_0(\Omega)$, one has the following equivalent definition 
\begin{equation}
\label{hardy!}
\mathfrak{h}_{s,p}(\Omega)=\inf_{u\in \widetilde{W}^{s,p}_0(\Omega)} \left\{[u]^p_{W^{s,p}(\mathbb{R}^N)}\, :\, \int_\Omega \frac{|u|^p}{d_\Omega^{s\,p}}\,dx=1\right\}.
\end{equation}
Indeed, the fact that the infimum over $\widetilde{W}^{s,p}_0(\Omega)$ is less than or equal to $\mathfrak{h}_{s,p}(\Omega)$ simply follows  from the fact that we enlarged the class of admissible functions.
\par 
To prove the converse inequality, we first observe that if $\mathfrak{h}_{s,p}(\Omega)=0$ then there is nothing to prove. If on the contrary $\mathfrak{h}_{s,p}(\Omega)>0$, then for every $u\in \widetilde{W}^{s,p}_0(\Omega)\setminus\{0\}$ we know that there exists a sequence $\{u_n\}_{n\in\mathbb{N}}\subset C^\infty_0(\Omega)$ converging to $u$ in $W^{s,p}(\mathbb{R}^N)$. Without loss of generality, we can assume that $u_n$ converges almost everywhere to $u$, as well. We then have
\[
[u]_{W^{s,p}(\mathbb{R}^N)}^p=\lim_{n\to\infty} [u_n]_{W^{s,p}(\mathbb{R}^N)}^p\ge \mathfrak{h}_{s,p}(\Omega)\,\liminf_{n\to\infty}\int_\Omega \frac{|u_n|^p}{d_\Omega^{s\,p}}\,dx\ge \mathfrak{h}_{s,p}(\Omega)\,\int_\Omega \frac{|u|^p}{d_\Omega^{s\,p}}\,dx,
\]
where we used Hardy's inequality in the first inequality and Fatou's Lemma in the second one.
\end{oss}

\begin{lm}
\label{lm:basics}
	Let $1<p<\infty$, $0<s<1$ and let $\Omega\subsetneq\mathbb{R}^N$ be an open set. If $\mathfrak{h}_{s,p}(\Omega)$ admits a non-trivial minimizer $u \in \widetilde{W}^{s,p}_0(\Omega)$, then this has constant sign in $\Omega$ and $u\not=0$ almost everywhere in $\Omega$. Moreover, the minimizer is unique, up to the choice of the sign and it is a weak solution of \eqref{equazione}, with 
$\lambda=\mathfrak{h}_{s,p}(\Omega)$.
	\end{lm}
\begin{proof}
Let us suppose that \eqref{hardy!} admits a minimizer $u\in\widetilde{W}^{s,p}_0(\Omega)$, in particular this implies that $\mathfrak{h}_{s,p}(\Omega)>0$. 	We observe that 
	\[
	\Big||a|-|b|\Big|\le |a-b|,\qquad \mbox{ for every } a,b\in\mathbb{R},
	\]
	and the inequality is strict, whenever $a\,b<0$. This yields 
\[
\mathfrak{h}_{s,p}(\Omega)\le \big[|u|\big]^p_{W^{s,p}(\mathbb{R})}\le [u]_{W^{s,p}(\mathbb{R})}^p=\mathfrak{h}_{s,p}(\Omega),
\]
and thus it must result 
\[
u(x)\,u(y)\ge 0,\qquad \mbox{ for a.\,e. }(x,y)\in\mathbb{R}^N\times\mathbb{R}^N.
\]
This shows that $u$ has constant sign almost everywhere in $\Omega$. Without loss of generality, we can suppose that $u$ is non-negative.
\par
We then observe that $u$ must be a minimizer of the following functional
\[
	\mathcal{F}(\varphi) := \frac{1}{p} \, [ \varphi ]_{W^{s,p}(\mathbb{R}^N)}^p - \frac{\mathfrak{h}_{s,p}(\Omega)}{p} \int_{\Omega} \frac{|\varphi|^p}{d_{\Omega}^{s\,p}} \, dx, \qquad \text{ for every } \varphi \in \widetilde{W}^{s,p}_0(\Omega),
\]
	as well.
Indeed, by definition of $\mathfrak{h}_{s,p}(\Omega)$, we have $\mathcal{F}(\varphi) \ge 0$ for every admissible function and $\mathcal{F}(u)=0$. Moreover, $u$ is non-trivial, due to the normalization on the weighted $L^p$ norm.
\par
By minimality, we get that $u$ must be a non-trivial non-negative weak solution of the Euler-Lagrange equation, which is given by \eqref{equazione} with $\lambda=\mathfrak{h}_{s,p}(\Omega)$.
By the minimum principle (see \cite[Theorem A.1]{BF}), we directly obtain that $u >0$ almost everywhere in $\Omega$, if the latter is connected. If $\Omega$ has more than one connected component, the same conclusion can be drawn by proceeding as in \cite[Proposition 2.6]{BP}, thanks to the nonlocality of the operator.
\par
We now show the uniqueness for the positive minimizer of $\mathfrak{h}_{s,p}(\Omega)$. For this, it is sufficient to exploit Lemma \ref{lm:FHC}.
Let us take $u, v \in \widetilde{W}^{s,p}_0(\Omega)$ two positive minimizers of $\mathfrak{h}_{s,p}(\Omega)$ and set
	\[ 
	\sigma= \left( \frac{1}{2}\, u^p + \frac{1}{2}\, v^p \right)^\frac{1}{p}, 
	\]
Thanks to \eqref{eq:hiddenconvexity}, we get that $\sigma\in\widetilde{W}^{s,p}_0(\Omega)$ is still a minimizer for $\mathfrak{h}_{s,p}(\Omega)$. 
	Thus \eqref{eq:hiddenconvexity} holds as an identity. By Lemma \ref{lm:FHC}, this means that
	there exists a constant $C$ such that
	\[ 
	u=C\,v,\qquad \text{ a.\,e. in } \Omega.
	\]  
	Finally, the normalization on the weighted norm implies that $C=1$. This concludes the proof.	
\end{proof}
\begin{oss}
In the local case, the uniqueness of an extremal for $\mathfrak{h}_{s,p}$ (provided it exists) can be found for example in \cite[Proposition 3.2]{MS}. Differently from \cite{MS}, here we found useful to rely on a hidden convexity argument, rather than on Picone's inequality.
\end{oss}

\begin{defi}
Let $\Omega\subsetneq\mathbb{R}^N$ be an open set. We say that $\partial\Omega$ is {\it locally continuous at} $x_0\in\partial\Omega$ if there exist:
\begin{itemize}
\item an open $N-$dimensional hyper-rectangle $Q_{\delta_0,\delta_1}$ centered at the origin, defined by
\[
Q_{\delta_0,\delta_1}=(-\delta_0,\delta_0)^{N-1}\times(-\delta_1,\delta_1),\qquad \mbox{ with } \delta_0,\delta_1>0;
\]
\vskip.2cm
\item a linear isometry $\mathcal O:\mathbb{R}^N\to\mathbb{R}^N$ such that $\mathcal{O}(x_0)=0$;
\vskip.2cm
\item a continuous function $\Psi:(-\delta_0,\delta_0)^{N-1}\to(-\delta_1,\delta_1)$;
\end{itemize}
such that
\[
Q_{\delta_0,\delta_1}\cap \mathcal{O}(\Omega)=\Big\{x=(x',x_N)\in Q_{\delta_0,\delta_1}\, :\, \Psi(x')< x_N< \delta_1\Big\},
\]
and
\[
Q_{\delta_0,\delta_1}\cap \mathcal{O}(\partial\Omega)=\Big\{x=(x',x_N)\in Q_{\delta_0,\delta_1}\, :\, x_N=\Psi(x')\Big\}.
\]
Roughly speaking, this means that $\partial\Omega$ coincides with the graph of a continuous function, in a small rectangular neighborhood of $x_0$.
\end{defi}

\begin{prop}
\label{prop:nonne}
	Let $1<p<\infty$, $0<s<1$ and let $\Omega\subsetneq\mathbb{R}^N$ be an open set, which is locally continuous at a point $x_0\in\partial\Omega$. Let us suppose that there exists a positive local weak supersolution $u$ of \eqref{equazione} with $\lambda=\mathfrak{h}_{s,p}(\Omega)$, such that
	\begin{equation}
	\label{dalbasso}
	u\ge \frac{1}{C}\, d_\Omega^\frac{s\,p-1}{p},\qquad \mbox{ in }\Omega.
	\end{equation}
	Then the infimum $\mathfrak{h}_{s,p}(\Omega)$ is not attained.
\end{prop}
\begin{proof}
	We first show that for such a set, we have 
\begin{equation}
\label{nonL1}
1/d_\Omega\not\in L^1(\Omega).
\end{equation}
At this aim, we can assume without loss of generality that 
\[
x_0=(0,\dots,0)\qquad \mbox{ and }\qquad \mathcal{O}=\mathrm{Id},
\]
so that 
\[
Q_{\delta_0,\delta_1}(x_0)\cap \Omega=\Big\{x=(x',x_N)\in Q_{\delta_0,\delta_1}(x_0)\, :\, \Psi(x')< x_N< \delta_1\Big\}.
\]
We then observe that (see Figure \ref{fig:continuity})
\[
d_\Omega(x)\le |x_N-\Psi(x')|=(x_N-\Psi(x')),\qquad \mbox{ for every } x=(x',x_N)\in Q_{\delta_0,\delta_1}(x_0)\cap \Omega.
\]
\begin{figure}
\includegraphics[scale=.35]{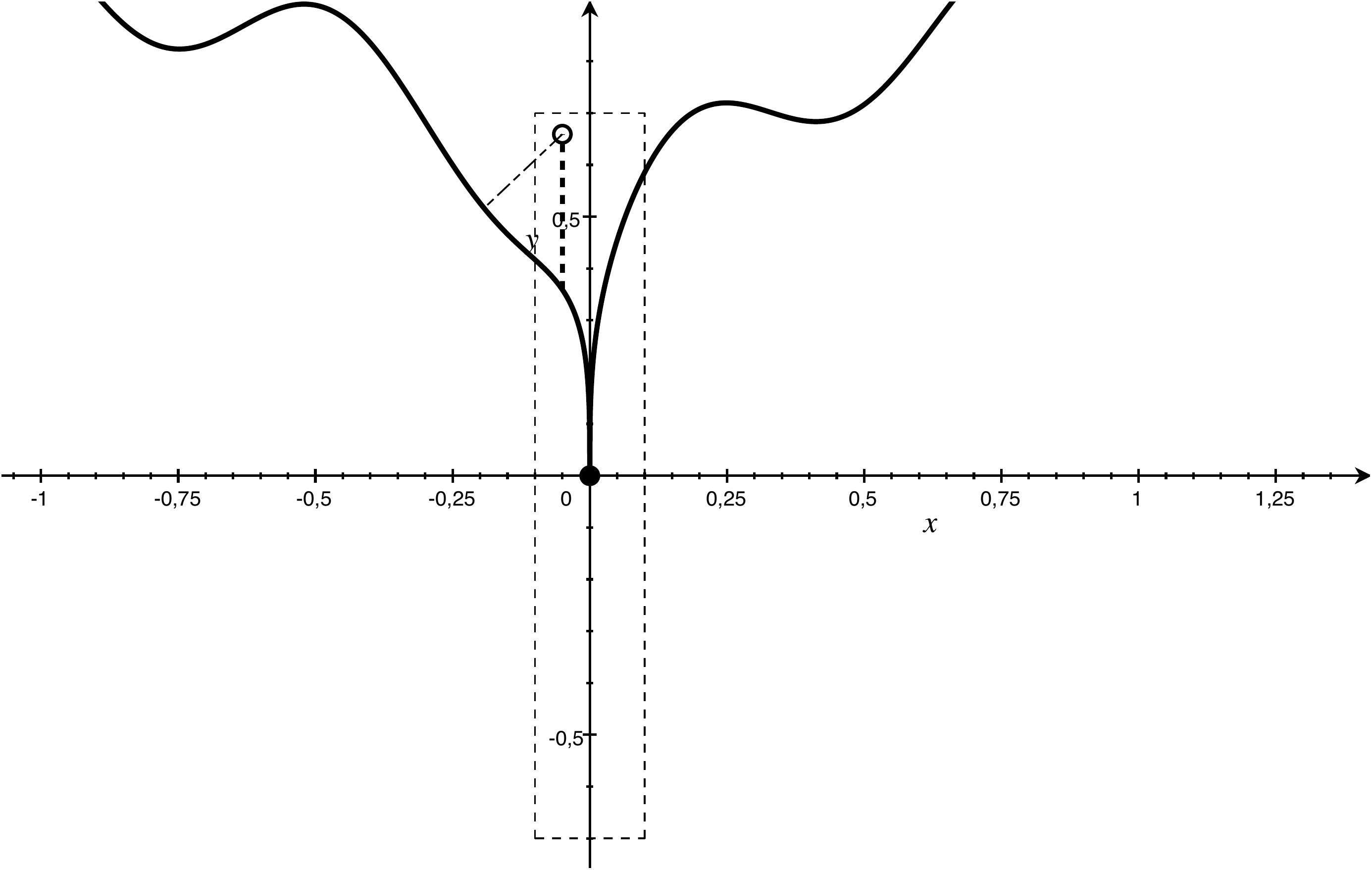}
\caption{For $(x',x_N)$ around a continuity point for the boundary, the ``vertical'' distance $x_N-\Psi(x')$ (in bold dashed line) is always larger than its distance from the boundary.}
\label{fig:continuity}
\end{figure}
This implies that 
\[
\int_\Omega \frac{1}{d_\Omega}\,dx\ge \int_{Q_{\delta_0,\delta_1}(x_0)\cap \Omega} \frac{1}{d_\Omega}\,dx\ge \int_{(-\delta_0,\delta_0)^{N-1}}\left(\int_{\Psi(x')}^{\delta_1} \frac{1}{x_N-\Psi(x')}\,dx_N\right)\,dx'.
\]
By observing that the last integral is diverging, we get \eqref{nonL1}.
\par
We now argue by contradiction and suppose that $v\in \widetilde{W}^{s,p}_0(\Omega)$ is a minimizer for $\mathfrak{h}_{s,p}(\Omega)$. This in particular implies that $\mathfrak{h}_{s,p}(\Omega)>0$. By Lemma \ref{lm:basics}, we can suppose that $v$ is positive. We then take a sequence $\{v_n\}_{n\in\mathbb{N}}\subset C^\infty_0(\Omega)$ approximating $v$ in $W^{s,p}(\mathbb{R}^N)$. Without loss of generality, we can take each $v_n$ to be non-negative and suppose that they converge to $v$ almost everywhere, as well.
	We then insert in the weak formulation of the equation for $u$ the test function
	\[
	\varphi=\frac{v_n^p}{u^{p-1}},
	\]
which is admissible thanks to Lemma \ref{lm:ammissibili1} and \eqref{dalbasso}. This leads to
	\begin{equation}
	\label{eq:stp}
	\begin{split}
	\iint_{\mathbb{R}^N\times\mathbb{R}^N}\frac{J_p(u(x)-u(y))}{|x-y|^{N+s\,p}}\,\left(\frac{v_n(x)^p}{u(x)^{p-1}}-\frac{v_n(y)^p}{u(y)^{p-1}}\right)\,dx\,dy\ge \mathfrak{h}_{s,p}(\Omega)\,\int_\Omega \frac{v_n^p}{d_\Omega^{s\,p}}\,dx.
	\end{split}
	\end{equation}
We now set
\[
\mathcal{R}(v_n,u):=|v_n(x)-v_n(y)|^p-J_p(u(x)-u(y))\,\left(\frac{v_n(x)^p}{u(x)^{p-1}}-\frac{v_n(y)^p}{u(y)^{p-1}}\right),
\]	
and observe that by Lemma \ref{lm:picone} this is always a non-negative quantity. With the previous notation, equation \eqref{eq:stp} can be rewritten as
\[
\mathfrak{h}_{s,p}(\Omega)\,\int_\Omega \frac{v_n^p}{d_\Omega^{s\,p}}\,dx+\iint_{\mathbb{R}^N\times\mathbb{R}^N} \frac{\mathcal{R}(v_n,u)}{|x-y|^{N+s\,p}}\,dx\,dy\le [v_n]^p_{W^{s,p}(\mathbb{R}^N)}.
\]
We now pass to the limit in the previous estimate and use Fatou's Lemma on the second term on the left-hand side: this yields
\[
\mathfrak{h}_{s,p}(\Omega)\,\int_\Omega \frac{v^p}{d_\Omega^{s\,p}}\,dx+\iint_{\mathbb{R}^N\times\mathbb{R}^N} \frac{\mathcal{R}(v,u)}{|x-y|^{N+s\,p}}\,dx\,dy\le [v]^p_{W^{s,p}(\mathbb{R}^N)}.
\]
By recalling that $v$ solves \eqref{hardy!}, the previous inequality gives
\[
\iint_{\mathbb{R}^N\times\mathbb{R}^N} \frac{\mathcal{R}(v,u)}{|x-y|^{N+s\,p}}\,dx\,dy=0.
\]
Since by Lemma \ref{lm:picone} we have $\mathcal{R}(v,u)\ge 0$ almost everywhere, this in turn implies that 
\[
0=\mathcal{R}(v,u)=|v(x)-v(y)|^p-J_p(u(x)-u(y))\,\left(\frac{v(x)^p}{u(x)^{p-1}}-\frac{v(y)^p}{u(y)^{p-1}}\right),\quad \mbox{ for a.\,e. } (x,y)\in \Omega\times\Omega.
\] 
By using the equality cases in the discrete Picone inequality, it follows that there exists a constant $C>0$ such that 
\[
u=C\,v,\qquad \mbox{ a.\,e. in }\Omega.
\]
This fact and the assumption \eqref{dalbasso} imply in particular that
	\[
	v\ge \frac{1}{C}\,d_\Omega^{\frac{sp-1}{p}},\qquad\mbox{ in } \Omega,
	\]
	possibly for a different constant $C>0$.
	By minimality of $v$, it follows 
	\[
	\begin{split}
	+\infty>[v]^p_{W^{s,p}(\mathbb{R}^N)}&= \mathfrak{h}_{s,p}(\Omega)\,\int_\Omega\frac{|v|^p}{d_\Omega^{s\,p}}\,dx\ge\frac{\mathfrak{h}_{s,p}(\Omega)}{C^p}\,\int_{\Omega}\frac{1}{d_\Omega}\,dx.
	\end{split}
	\]
This finally gives a contradiction with \eqref{nonL1}.	
\end{proof}

\section{Construction of supersolutions in dimension $1$}
\label{sec:4}

\subsection{The half-line}
In what follows, for $t>0$ we use the notation 
\begin{equation}
\label{intervallino}
I_\varepsilon(t):=\Big((1-\varepsilon)\,t,(1+\varepsilon)\,t\Big),\qquad \mbox{ for } 0<\varepsilon\ll 1.
\end{equation}
We still use the notation $\mathbb{H}^1_+=(0,+\infty)$.
Let $\beta\in\mathbb{R}$,
we set 
\[
U_\beta(t):=t^\beta,\qquad \mbox{ for } t\in\mathbb{H}^1_+,
\]
and extend it by $0$ to the complement of $\mathbb{H}^1_+$. In particular, in the borderline case $\beta=0$, this has to be intended as the characteristic function of $\mathbb{H}^1_+$. 
\par
The next result collects some properties of $U_\beta$ which will be useful in the sequel.
\begin{lm}
\label{lm:potenzesobolev}
Let $1<p<\infty$ and $0<s<1$. For every $\beta\in \mathbb{R}$
we have $U_\beta\in W^{s,p}_{\rm loc}(\mathbb{H}^1_+)$.
Moreover, $U_\beta$ has the following further properties:
\begin{itemize}
\item  for  
\[
\frac{s\,p-1}{p}<\beta,
\]
we have $U_\beta\in W^{s,p}((0,M))$, for every $M>0$;
\vskip.2cm
\item for
\[
-\frac{1}{p-1}<\beta<\frac{s\,p}{p-1},
\] 
we have $U_\beta\in L^{p-1}_{s\,p}(\mathbb{R})$. 
\end{itemize}
\end{lm}
\begin{proof}
We observe that $U_\beta$ is locally Lipschitz on $\mathbb{H}^1_+$, for every $\beta\in\mathbb{R}$. This easily implies that $U_\beta\in W^{s,p}_{\rm loc}(\mathbb{H}^1_+)$.
\vskip.2cm\noindent
Let us now suppose that $\beta>(s\,p-1)/p$. From the fact that $U_\beta\in W^{s,p}_{\rm loc}(\mathbb{H}^1_+)$, we get that for every $0<\varepsilon<M$ we have 
\[
\int_{\varepsilon}^M\int_{\varepsilon}^M \frac{|U_\beta(t)-U_\beta(y)|^p}{|t-y|^{1+s\,p}}\,dt\,dy<+\infty.
\]
We show that this is uniformly bounded with respect to $\varepsilon$. For $\beta>s$ this is straightforward, it is sufficient to use that $U_\beta$ is either $\beta-$H\"older continuous (for $s<\beta<1$) or even Lipschitz continuous (for $\beta\ge 1$) on $[0,M]$.
\par
We thus assume $(s\,p-1)/p<\beta\le s$. By using the definition of $U_\beta$, Fubini's Theorem and the change of variable $y=\tau\,t$, we get
\begin{equation}
\label{babbuaggini}
\begin{split}
\int_{\varepsilon}^M\int_{\varepsilon}^M \frac{|U_\beta(t)-U_\beta(y)|^p}{|t-y|^{1+s\,p}}\,dt\,dy&=\int_{\varepsilon}^M\left(\int_\frac{\varepsilon}{t}^\frac{M}{t} \frac{|1-\tau^\beta|^p}{|1-\tau|^{1+s\,p}}\,d\tau\right)\,t^{\beta\,p-s\,p}\,dt\\
&= \int_{\varepsilon}^M\left(\int_\frac{\varepsilon}{t}^1 \frac{|1-\tau^\beta|^p}{|1-\tau|^{1+s\,p}}\,d\tau\right)\,t^{\beta\,p-s\,p}\,dt\\
&+\int_{\varepsilon}^M\left(\int_1^\frac{M}{t} \frac{|1-\tau^\beta|^p}{|1-\tau|^{1+s\,p}}\,d\tau\right)\,t^{\beta\,p-s\,p}\,dt.
\end{split}
\end{equation}
We now observe that 
\[
\int_\frac{\varepsilon}{t}^1 \frac{|1-\tau^\beta|^p}{|1-\tau|^{1+s\,p}}\,d\tau\le \int_0^1 \frac{|1-\tau^\beta|^p}{|1-\tau|^{1+s\,p}}\,d\tau<+\infty.
\]
For second integral, we observe that 
\[
\frac{|1-\tau^\beta|^p}{|1-\tau|^{1+s\,p}}\sim \frac{1}{\tau^{1+s\,p-\beta\,p}},\qquad \mbox{ for } \tau\to +\infty,
\]
and the last function is integrable on $[1,+\infty)$, for $\beta<s$. Thus we get 
\[
\int_1^\frac{M}{t} \frac{|1-\tau^\beta|^p}{|1-\tau|^{1+s\,p}}\,d\tau\le \int_1^{+\infty} \frac{|1-\tau^\beta|^p}{|1-\tau|^{1+s\,p}}\,d\tau<+\infty.
\]
This discussion entails that 
\[
\int_{\varepsilon}^M\int_{\varepsilon}^M \frac{|U_\beta(t)-U_\beta(y)|^p}{|t-y|^{1+s\,p}}\,dt\,dy\le C\,\int_{\varepsilon}^Mt^{\beta\,p-s\,p}\,dt=C\,\frac{M^{\beta\,p-s\,p+1}-\varepsilon^{\beta\,p-s\,p+1}}{\beta\,p-s\,p+1},
\]
and the last quantity is bounded as $\varepsilon$ goes to $0$, thanks to the fact that $\beta>(s\,p-1)/p$. We thus proved the claimed property of $U_\beta$, for $(s\,p-1)/p<\beta<s$.
\par
We still miss the borderline case $\beta=s$. From \eqref{babbuaggini}, we can infer
\[
\begin{split}
\int_{\varepsilon}^M\int_{\varepsilon}^M \frac{|U_s(t)-U_s(y)|^p}{|t-y|^{1+s\,p}}\,dt\,dy&\le \int_{\varepsilon}^M\left(\int_{0}^1 \frac{|1-\tau^s|^p}{|1-\tau|^{1+s\,p}}\,d\tau\right)\,dt+\int_{\varepsilon}^M\left(\int_1^\frac{M}{t} \frac{|1-\tau^s|^p}{|1-\tau|^{1+s\,p}}\,d\tau\right)\,dt.
\end{split}
\]
The first integral on the right-hand side is uniformly bounded in $\varepsilon$, but now we have to pay attention to the fact that 
\[
\lim_{t\to 0^+}\int_1^\frac{M}{t} \frac{|1-\tau^s|^p}{|1-\tau|^{1+s\,p}}\,d\tau=+\infty.
\]
We can proceed as follows: we write
\[
\begin{split}
\int_{\varepsilon}^M\left(\int_1^\frac{M}{t} \frac{|1-\tau^s|^p}{|1-\tau|^{1+s\,p}}\,d\tau\right)\,dt&=\int_{\varepsilon}^\frac{M}{2}\left(\int_1^\frac{M}{t} \frac{|1-\tau^s|^p}{|1-\tau|^{1+s\,p}}\,d\tau\right)\,dt+\int_\frac{M}{2}^M\left(\int_1^\frac{M}{t} \frac{|1-\tau^s|^p}{|1-\tau|^{1+s\,p}}\,d\tau\right)\,dt
\end{split}
\]
and observe that for $0<t<M/2$, we have
\[
\int_{\varepsilon}^\frac{M}{2}\left(\int_1^\frac{M}{t} \frac{|1-\tau^s|^p}{|1-\tau|^{1+s\,p}}\,d\tau\right)\,dt\le \frac{M}{2}\,\int_1^2 \frac{|1-\tau^s|^p}{|1-\tau|^{1+s\,p}}\,d\tau+\int_{\varepsilon}^\frac{M}{2}\left(\int_2^\frac{M}{t} \frac{|1-\tau^s|^p}{|1-\tau|^{1+s\,p}}\,d\tau\right)\,dt
\]
and, at last
\[
\int_{\varepsilon}^\frac{M}{2}\left(\int_2^\frac{M}{t} \frac{|1-\tau^s|^p}{|1-\tau|^{1+s\,p}}\,d\tau\right)\,dt\le 2^{1+s\,p}\,\int_{\varepsilon}^\frac{M}{2}\left(\int_2^\frac{M}{t} \tau^{p\,s-1-s\,p}\,d\tau\right)\,dt= 2^{1+s\,p}\,\int_\varepsilon^M \log\left(\frac{M}{2\,t}\right)\,dt.
\]
The last integral is uniformly bounded, as $\varepsilon$ goes to $0$. This finally proves that $U_s\in W^{s,p}((0,M))$.
\vskip.2cm\noindent
Finally, we observe that 
\[
U_\beta\in L^{p-1}_{\rm loc}(\mathbb{R})\qquad \Longleftrightarrow \qquad \beta\,(p-1)>-1,
\]
and 
\[
\int_{\mathbb{R}} \frac{U_\beta^{p-1}}{(1+|t|)^{1+s\,p}}\,dt=\int_0^{+\infty} \frac{t^{\beta\,(p-1)}}{(1+t)^{1+s\,p}}\,dt<+\infty\qquad \Longleftrightarrow \qquad -1<\beta\,(p-1)<s\,p.
\]
This concludes the proof.
\end{proof}
\begin{oss}
\label{oss:miserve!}
For later reference, we observe that in the previous proof for 
\[
\frac{s\,p-1}{p}<\beta<s,
\]
 we proved the following upper bound
\[
[U_\beta]^p_{W^{s,p}((0,M))}\le \left(\int_0^1 \frac{|1-\tau^\beta|^p}{|1-\tau|^{1+s\,p}}\,d\tau+\int_1^{+\infty} \frac{|1-\tau^\beta|^p}{|1-\tau|^{1+s\,p}}\,d\tau\right)\frac{M^{\beta\,p-s\,p+1}}{\beta\,p-s\,p+1}.
\]
By making the change of variable $\tau=1/\xi$ in the second integral, this can also be rewritten as 
\begin{equation}
\label{stima1dsharp}
[U_\beta]^p_{W^{s,p}((0,M))}\le \left(\int_0^1 \frac{|1-\tau^\beta|^p}{|1-\tau|^{1+s\,p}}\,\left(1+\tau^{s\,p-\beta\,p-1}\right)\,d\tau\right)\frac{M^{\beta\,p-s\,p+1}}{\beta\,p-s\,p+1}.
\end{equation}
\end{oss}
In the next result, we compute the fractional $p-$Laplacian of order $s$ for $U_\beta$. This generalizes \cite[Lemma 3.1]{IMS} to the case $\beta\not=s$.
\begin{prop}
\label{prop:ACF}
Let $1<p<\infty$ and $0<s<1$. For every
\[
-\frac{1}{p-1}<\beta<\frac{s\,p}{p-1},
\] 
the function $U_\beta$ is a local weak solution of \eqref{equazione} in $\mathbb{H}^1_+$, with 
\begin{equation}
\label{lambda}
\lambda=\lambda(\beta)=2\,\int_0^1 \dfrac{J_p(1-t^\beta)}{(1-t)^{1+s\,p}}\,\left(1-t^{s\,p-1-\beta\,(p-1)}\right)\,dt+\dfrac{2}{s\,p}.
\end{equation}
Moreover, if we define the family of functions on $\mathbb{H}^1_+$ by
\begin{equation}
\label{Fe}
F_\varepsilon(t)=2\,\int_{\mathbb{R}\setminus I_\varepsilon(t)} \frac{J_p(U_\beta(t)-U_\beta(y))}{|t-y|^{1+s\,p}}\,dy,\qquad \mbox{ for } 0<\varepsilon< 1,
\end{equation}
where $I_\varepsilon(t)$ is defined by \eqref{intervallino}, we get that this converges to 
\[
F_0(t)=\lambda(\beta)\,\frac{U_\beta(t)^{p-1}}{t^{s\,p}},
\]
uniformly on compact subsets of $\mathbb{H}^1_+$, as $\varepsilon$ goes to $0$.
\end{prop}

\begin{proof}
Let us take $\varphi\in C^\infty_0(\mathbb{H}^1_+)$, we observe that by the Dominated Convergence Theorem we have 
\[
\begin{split}
\iint_{\mathbb{R}\times\mathbb{R}}&\frac{J_p(U_\beta(t)-U_\beta(y))\,(\varphi(t)-\varphi(y))}{|t-y|^{1+s\,p}}\,dt\,dy=\lim_{\varepsilon\to 0^+} \iint_{(\mathbb{R}\times\mathbb{R})\setminus \mathcal{O}_\varepsilon}\!\!\frac{J_p(U_\beta(t)-U_\beta(y))\,(\varphi(t)-\varphi(y))}{|t-y|^{1+s\,p}}\,dt\,dy,\\
\end{split}
\]
where
\[
\mathcal{O}_\varepsilon=\Big\{(t,y)\in\mathbb{R}\times\mathbb{R}\, :\, \min\{(1-\varepsilon)\,t, (1+\varepsilon)\,t \} \le y \le \max\{(1-\varepsilon)\,t, (1+\varepsilon)\,t \} \Big\},
\]
see Figure \ref{fig:conical}.
\begin{figure}
\includegraphics[scale=.25]{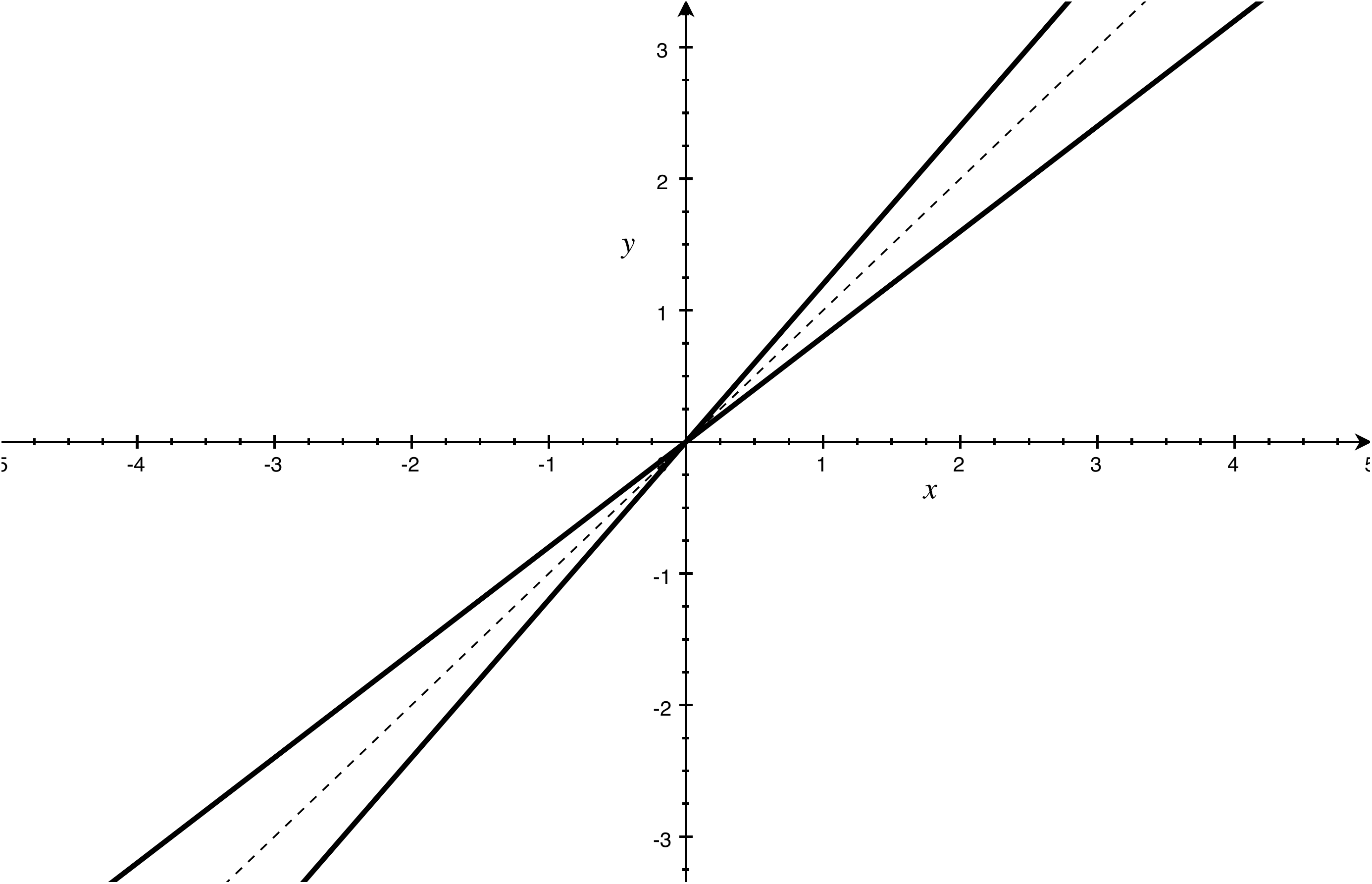}
\caption{The set $\mathcal{O}_\varepsilon$ is the conical region ``centered'' around the line $y=t$.}
\label{fig:conical}
\end{figure}
For every $0<\varepsilon<1$, by proceeding as in \cite[Lemma 2.3]{BC}, 
we have 
\[
\frac{J_p(U_\beta(t)-U_\beta(y))}{|t-y|^{1+s\,p}}\,\varphi(t)\in L^1((\mathbb{R}\times\mathbb{R})\setminus \mathcal{O}_\varepsilon).
\]
Thus we can use Fubini's Theorem and a change of variable, to write
\[
\begin{split}
\iint_{(\mathbb{R}\times\mathbb{R})\setminus \mathcal{O}_\varepsilon}&\frac{J_p(U_\beta(t)-U_\beta(y))\,(\varphi(t)-\varphi(y))}{|t-y|^{1+s\,p}}\,dt\,dy=2\,\int_{0}^{+\infty} \left(\int_{\mathbb{R}\setminus I_\varepsilon(t)}\frac{J_p(U_\beta(t)-U_\beta(y))}{|t-y|^{1+s\,p}}\,dy\right)\,\varphi(t)\,dt.\\
\end{split}
\]
Observe that we used that $\varphi$ is compactly supported on $\mathbb{H}^1_+$.
By recalling the definition \eqref{Fe}, up to now we have obtained 
\begin{equation}
\label{Fe2}
\iint_{\mathbb{R}\times\mathbb{R}}\frac{J_p(U_\beta(t)-U_\beta(y))\,(\varphi(t)-\varphi(y))}{|t-y|^{1+s\,p}}\,dt\,dy=\lim_{\varepsilon\to 0^+} \int_\mathbb{R} F_\varepsilon(t)\,\varphi(t)\,dt,
\end{equation}
for every $\varphi\in C^\infty_0(\mathbb{H}^1_+)$. We now manipulate this quantity, for a fixed $0<\varepsilon<1$: by recalling that $U_\beta$ identically vanishes in $(-\infty,0]$, for $t>0$ we have
	\[
	\begin{split}
	F_\varepsilon(t)&=2 \,\int_{\mathbb{R} \setminus I_\varepsilon(t)}\frac{J_p(U_\beta(t)-U_\beta(y))}{|t-y|^{1+s\,p}}\,dy\\
	&=2\,\int_{\mathbb{H}^1_+ \setminus I_\varepsilon(t)}\frac{J_p(t^\beta-y^\beta)}{|t-y|^{1+s\,p}}\,dy+2\,\int^0_{-\infty}\frac{t^{\beta\,(p-1)}}{|t-y|^{1+s\,p}}\,dy.
	\end{split}
	\]
	The second integral can be directly computed: this gives
	\[
\int^0_{-\infty}\frac{t^{\beta\,(p-1)}}{|t-y|^{1+s\,p}}\,dy
= \frac{1}{s\,p}\,\frac{U_\beta(t)^{p-1}}{t^{s\,p}},
	\]
	where we used the definition of $U_\beta(t)$. For the first integral in the definition of $F_\varepsilon$, by performing the change of variable $y=\tau\,t$, we obtain
	\[
	\begin{split}
	\int_{\mathbb{H}^1_+ \setminus I_\varepsilon(t)}\frac{J_p(t^\beta-y^\beta)}{|t-y|^{1+s\,p}}\,dy &=\frac{t^{\beta\,(p-1)}}{t^{s\,p}}\,\int_0^{1-\varepsilon}\frac{J_p(1-\tau^\beta)}{|1-\tau|^{1+s\,p}}\,d\tau +\frac{t^{\beta\,(p-1)}}{t^{s\,p}}\,\int_{1+\varepsilon}^{+\infty}\frac{J_p(1-\tau^\beta)}{|1-\tau|^{1+s\,p}}\,d\tau \\
	&=\frac{U_\beta(t)^{p-1}}{t^{s\,p}}\,\left(\int_0^{1-\varepsilon}\frac{J_p(1-\tau^\beta)}{|1-\tau|^{1+s\,p}}\,d\tau +\int_{1+\varepsilon}^{+\infty}\frac{J_p(1-\tau^\beta)}{|1-\tau|^{1+s\,p}}\,d\tau \right), 
	\end{split}
	\]
again thanks to the definition of $U_\beta$. Thus we have obtained 
\begin{equation}
\label{beh!}
F_\varepsilon(t)=\lambda_\varepsilon(\beta)\,\frac{U_\beta(t)^{p-1}}{t^{s\,p}},\qquad \mbox{ for every } t\in \mathbb{H}^1_+,\, 0<\varepsilon<1,
\end{equation}	
where 
\[
\lambda_\varepsilon(\beta)=2\,\int_0^{1-\varepsilon}\frac{J_p(1-\tau^\beta)}{|1-\tau|^{1+s\,p}}\,d\tau +2\,\int_{1+\varepsilon}^{+\infty}\frac{J_p(1-\tau^\beta)}{|1-\tau|^{1+s\,p}}\,d\tau +\frac{2}{s\,p}.
\]
By inserting this in \eqref{Fe2}, we have 
\begin{equation}
\label{Fe3}
\iint_{\mathbb{R}\times\mathbb{R}}\frac{J_p(U_\beta(t)-U_\beta(y))\,(\varphi(t)-\varphi(y))}{|t-y|^{1+s\,p}}\,dt\,dy=\left(\lim_{\varepsilon\to 0^+}  \lambda_\varepsilon(\beta)\right)\,\int_\mathbb{R}\frac{U_\beta(t)^{p-1}}{t^{s\,p}}\,\varphi(t)\,dt.
\end{equation}
	To conclude the proof, we only need to show that for $\lambda(\beta)$ defined by \eqref{lambda}, we have 
	\[ 
	\lambda(\beta)= \lim_{\varepsilon \to 0^+} \lambda_\varepsilon(\beta),\qquad \mbox{ for every } -\frac{1}{p-1}<\beta<\frac{s\,p}{p-1}.
	\]
We first observe that the case $\beta=0$ is simple: in this case we have 
\[
J_p(t^\beta-1)=0,\qquad \mbox{ for } t\in(0,1),
\]	
and thus we directly get 
\[
\lambda(0)=\lambda_\varepsilon(0)=\frac{2}{s\,p}.
\]
We can thus suppose that $\beta\not =0$. By recalling the definition of $\lambda_\varepsilon(\beta)$ above and performing the change of variable $\tau=1/\zeta$ in the second integral, we get for $0<\varepsilon<1$
\[
\begin{split}
\lambda_\varepsilon(\beta)&=2\,\int_0^{1-\varepsilon}\frac{J_p(1-\tau^\beta)}{|1-\tau|^{1+s\,p}}\,d\tau +2\,\int_0^\frac{1}{1+\varepsilon}\frac{J_p(1-\zeta^{-\beta})}{|1-\zeta^{-1}|^{1+s\,p}}\,\frac{d\zeta}{\zeta^2} +\frac{2}{s\,p}\\
&=2\,\int_0^{1-\varepsilon}\frac{J_p(1-\tau^\beta)}{|1-\tau|^{1+s\,p}}\,d\tau +2\,\int_0^\frac{1}{1+\varepsilon}\frac{J_p(\zeta^\beta-1)}{|\zeta-1|^{1+s\,p}}\,\zeta^{s\,p-1-\beta\,(p-1)}d\zeta+\frac{2}{s\,p}.
\end{split}
\]
By observing that 
\[
1-\varepsilon<\frac{1}{1+\varepsilon},\qquad \mbox{ for } 0<\varepsilon<1,
\]
we can write
\[
\begin{split}
\lambda_\varepsilon(\beta)&=2\,\int_0^{1-\varepsilon}\frac{J_p(1-\tau^\beta)}{|1-\tau|^{1+s\,p}}\,(1-\tau^{s\,p-1-\beta\,(p-1)})\,d\tau +2\,\int_{1-\varepsilon}^\frac{1}{1+\varepsilon}\frac{J_p(\tau^\beta-1)}{|\tau-1|^{1+s\,p}}\,\tau^{s\,p-1-\beta\,(p-1)}\,d\tau+\frac{2}{s\,p}.
\end{split}
\]
We claim that 
\begin{equation}
\label{azzero}
\lim_{\varepsilon\to 0^+} \int_{1-\varepsilon}^\frac{1}{1+\varepsilon}\frac{J_p(\tau^\beta-1)}{|\tau-1|^{1+s\,p}}\,\tau^{s\,p-1-\beta\,(p-1)}\,d\tau=0.
\end{equation}
Observe at first that for $0<\varepsilon<1/2$, we have
\[
\begin{split}
\tau^{s\,p-1-\beta\,(p-1)}&\le \max\left\{(1-\varepsilon)^{s\,p-1-\beta\,(p-1)},\left(\frac{1}{1+\varepsilon}\right)^{s\,p-1-\beta\,(p-1)}\right\}\\
&\le \max\left\{2^{-s\,p+1+\beta\,(p-1)},1\right\}=C,\qquad \mbox{ for } 1-\varepsilon<\tau<\frac{1}{1+\varepsilon}.
\end{split}
\]
Thus we get
\[
\begin{split}
\left|\int_{1-\varepsilon}^\frac{1}{1+\varepsilon}\frac{J_p(\tau^\beta-1)}{|\tau-1|^{1+s\,p}}\,\tau^{s\,p-1-\beta\,(p-1)}\,d\tau\right|&\le C\,\int_{1-\varepsilon}^\frac{1}{1+\varepsilon}\frac{|\tau^\beta-1|^{p-1}}{|\tau-1|^{1+s\,p}}\,d\tau.\\
\end{split}
\]
By using Lemma \ref{lm:lemmino}, we can further estimate for $0<\varepsilon<1/2$
\[
\begin{split}
\left|\int_{1-\varepsilon}^\frac{1}{1+\varepsilon}\frac{J_p(\tau^\beta-1)}{|\tau-1|^{1+s\,p}}\,\tau^{s\,p-1-\beta\,(p-1)}\,d\tau\right|
&\le C\,|\beta|^{p-1}\,\int_{1-\varepsilon}^\frac{1}{1+\varepsilon} \max\big\{\tau^{\beta-1},1\big\}^{p-1}\,(1-\tau)^{p-2-s\,p}\,d\tau\\
&\le C\,|\beta|^{p-1}\,\max\big\{2^{1-\beta},1\big\}^{p-1}\,\int_{1-\varepsilon}^\frac{1}{1+\varepsilon} (1-\tau)^{p-2-s\,p}\,d\tau.
\end{split}
\]
By a direct computation, we now get 
\[
\lim_{\varepsilon \to 0^+}\int_{1-\varepsilon}^\frac{1}{1+\varepsilon} (1-\tau)^{p-2-s\,p}\,d\tau=0,
\]
which in turn implies \eqref{azzero}. On the other hand, by a Taylor expansion, we have 
\[
\frac{J_p(1-\tau^\beta)}{|1-\tau|^{1+s\,p}}\,(1-\tau^{s\,p-1-\beta\,(p-1)})\sim \beta^{p-1}\,\big(s\,p-1-\beta\,(p-1)\big)\,(1-\tau)^{p\,(1-s)-1},\qquad \mbox{ for }\tau \nearrow 1^-,
\]
which shows that 
\[
\frac{J_p(1-\tau^\beta)}{|1-\tau|^{1+s\,p}}\,(1-\tau^{s\,p-1-\beta\,(p-1)})\in L^1((0,1)).
\]
These facts permit to establish that
	\[ 
	\lambda(\beta)= \lim_{\varepsilon \to 0^+} \lambda_\varepsilon(\beta),\qquad \mbox{ for every } -\frac{1}{p-1}<\beta<\frac{s\,p}{p-1},
	\]
thus from \eqref{Fe3} we get that $U_\beta$ is a local weak solution of the claimed equation.
\par
The last statement about the convergence of $F_\varepsilon$ is an easy consequence of formula \eqref{beh!}.
\end{proof}

The next result investigates some properties of the function $\lambda(\beta)$ defined in \eqref{lambda}.
This in particular permits to single out a special solution, among all the functions $U_\beta$: this corresponds to the choice
\[
\beta=\frac{s\,p-1}{p}.
\]
Indeed, for this function, the constant $\lambda$ is {\it the largest possible}. This extends to $1<p<\infty$ a similar discussion contained in the proof of \cite[Theorem 1]{BD}.
\begin{prop}
\label{lm:costantibeta}
Let $1<p<\infty$ and $0<s<1$. Let us consider the function
\[
\beta\mapsto \lambda(\beta),\qquad \mbox{ defined by \eqref{lambda} on the interval }\left(-\dfrac{1}{p-1},\dfrac{s\,p}{p-1}\right).
\]
Then this has the following properties:
\begin{enumerate}
\item it is monotonically decreasing for $\beta>(s\,p-1)/p$ and monotonically increasing for $\beta<(s\,p-1)/p$. In particular, we have
\[
\lambda(\beta)\le \lambda\left(\frac{s\,p-1}{p}\right)=2\,\int_0^1 \dfrac{\left|1-t^\frac{s\,p-1}{p}\right|^p}{(1-t)^{1+s\,p}}\,dt+\dfrac{2}{s\,p};
\]
\item there exists $\beta^*=\beta^*(s,p)$ such that
\[
-\frac{1}{p-1}<\beta^*<\frac{s\,p-1}{p}\qquad \mbox{ and }\qquad \lambda(\beta^*)=\lambda(s)=0.
\]
In particular, we have 
\[
\lambda(\beta)\ge0 \qquad \Longleftrightarrow \qquad \beta^*\le \beta\le s.
\]
\end{enumerate}
\end{prop}
\begin{proof}
We proceed similarly as in \cite[Lemma B.1]{BF}, but making a more complete study.
For every $0<t<1$, we consider the function defined by
\[
g(\beta)=J_p(1-t^\beta)\,\left(1-t^{s\,p-1-\beta\,(p-1)}\right),\qquad \mbox{ for } 0\le \beta<\frac{s\,p}{p-1}.
\]
We discuss the monotonicity of such a function. We first observe that $g(0)=0$. Let us take $\beta\not =0$ and
differentiate $g$, we have
\begin{equation}
\label{derivata}
\begin{split}
g'(\beta)&=(p-1)\,t^{s\,p-1-\beta\,(p-1)}\,\log t\, J_p(1-t^\beta)\\
&-(p-1)\,t^\beta\, \log t\,\left(1-t^{s\,p-1-\beta\,(p-1)}\right)\,J_p'(1-t^\beta),
\end{split}
\end{equation}
By observing that $\log t<0$, $J_p'(1-t^\beta)>0$ and that $J_p(1-t^\beta)>0$ for $\beta>0$, we get
\[
\begin{split}
g'(\beta)\ge 0\quad& \Longleftrightarrow\quad t^{s\,p-1-\beta\,(p-1)}\,(1-t^\beta)-t^\beta\left(1-t^{s\,p-1-\beta\,(p-1)}\right)\le 0\\
& \Longleftrightarrow\quad t^{s\,p-1-\beta\,(p-1)}\le t^\beta\\
&\Longleftrightarrow \quad t^{s\,p-1-\beta\,p}\le 1.
\end{split}
\]
Since $0<t<1$, the last requirement is equivalent to 
\[
\beta\le \frac{s\,p-1}{p}.
\]
This implies that: 
\begin{itemize}
\item if $s\,p<1$, then $g$ is monotone decreasing on the whole interval $(0,(s\,p)/(p-1))$ and thus 
\[
g(\beta)\le g(0)=0,\qquad \mbox{ for every } 0\le \beta<\frac{s\,p}{p-1};
\]
\item if $s\,p\ge1$, then $g$ is monotone increasing on $(0,(s\,p-1)/p)$ and monotone decreasing on $((s\,p-1)/p,s\,p/(p-1))$. In particular, it is maximal at $\beta=(s\,p-1)/p$ and thus
\[
g(\beta)\le g\left(\frac{s\,p-1}{p}\right)=\left(1-t^\frac{s\,p-1}{p}\right)^p,\qquad \mbox{ for every } 0\le \beta<\frac{s\,p}{p-1}.
\]
We also observe that 
\[
\left(1-t^\frac{s\,p-1}{p}\right)^p\ge 0.
\]
\end{itemize}
We now perform a similar discussion for $\beta<0$: from \eqref{derivata}, by noticing that this time $J_p(1-t^\beta)=-J_p(t^\beta-1)<0$, while $J_p'(1-t^\beta)=J_p'(t^\beta-1)$, 
we get again
\[
g'(\beta)\ge 0\quad \Longleftrightarrow \quad \beta\le \frac{s\,p-1}{p}.
\]
As above, this implies that: 
\begin{itemize}
\item if $s\,p< 1$, then $g$ is monotone increasing on $(-1/(p-1),(s\,p-1)/p)$ and monotone decreasing on $((s\,p-1)/p,0)$. In particular, it is maximal at $\beta=(s\,p-1)/p$ and thus
\[
g(\beta)\le g\left(\frac{s\,p-1}{p}\right)=\left(t^\frac{s\,p-1}{p}-1\right)^p,\qquad \mbox{ for every } -\frac{1}{p-1}<\beta<0.
\]
We also observe that 
\[
\left(t^\frac{s\,p-1}{p}-1\right)^p\ge 0;
\]
\item if $s\,p\ge 1$, then $g$ is monotone increasing on the whole interval $(-1/(p-1),0)$ and thus
\[
g(\beta)\le \lim_{\tau\to 0^-}g(\tau)=0,\qquad \mbox{ for every } -\frac{1}{p-1}< \beta<0.
\]
\end{itemize}
In particular, this finally permits to infer that 
\[
\max_{-\frac{1}{p-1}<\beta<\frac{s\,p}{p-1}} g(\beta)\le \left|1-t^\frac{s\,p-1}{p}\right|^p.
\]
and such a maximal value is uniquely attained at $\beta=(s\,p-1)/p$.
By recalling that by definition 
\[
\lambda(\beta)=2\,\int_0^1\frac{g(\beta)}{(1-t)^{1+s\,p}}\,dt+\frac{2}{s\,p},
\]
the properties of $\lambda$ claimed in (1) follow from the above detailed discussion on $g$.
\vskip.2cm\noindent
Finally, the fact that $\lambda(s)=0$ has been proved in \cite[Lemma 3.1]{IMS}. The existence of the exponent $\beta^*$ now follows by using the monotonicity and continuity of $\lambda$, together with the fact that
\[
\lambda\left(\frac{s\,p-1}{p}\right)>0\qquad \mbox{ and }\qquad \lim_{\beta\to\left(-\frac{1}{p-1}\right)^{+}}\lambda(\beta)=-\infty.
\]
This concludes the proof.
\end{proof}
\begin{oss}[The exponent $\beta^*$]
\label{oss:exponent}
For $p=2$, the function $\lambda(\beta)$ is given by 
\[
\lambda(\beta)=\int_0^1 \dfrac{1-t^\beta}{(1-t)^{1+2\,s}}\,\left(1-t^{2\,s-1-\beta}\right)\,dt+\dfrac{1}{s}.
\]
It is not difficult to see that such a function is symmetric with respect to the maximum point $(2\,s-1)/2$, i.e. we have 
\begin{equation}
\label{symmetric}
\lambda(2\,s-1-\beta)=\lambda(\beta),\qquad \mbox{ for every } -1<\beta<2\,s.
\end{equation}
Accordingly, the exponent $\beta^*$ in this case is simply given by 
\[
\beta^*=2\,s-1-s=s-1.
\]
With such a choice, in view of \eqref{symmetric}, we have
\[
\lambda(s-1)=\lambda(s)=0.
\]
Another case where $\beta^*$ can be explicitly determined is when $s\,p=1$. In this case, we have 
\[
\lambda(\beta)=\int_0^1 \dfrac{J_p(1-t^\beta)}{(1-t)^{2}}\,\left(1-t^{-\beta\,(p-1)}\right)\,dt+2,
\]
and we observe that 
\[
\begin{split}
J_p(1-t^\beta)\,\left(1-t^{-\beta\,(p-1)}\right)
=J_p(1-t^{-\beta})\left(1-t^{\beta\,(p-1)}\right),
\end{split}
\]
thanks to the oddness and the homogeneity of $J_p$. This shows that $\beta\mapsto \lambda(\beta)$ is an even function, i.e.
\[
\lambda(\beta)=\lambda(-\beta),\qquad \mbox{ for every } -\frac{1}{p-1}<\beta<\frac{1}{p-1}.
\]
Thus, by recalling that $\lambda(s)=0$, we get in this case that $\beta^*=-s=-1/p$. 
\end{oss}

\subsection{The interval for $p=2$}
By using the results of the previous subsection and the properties of the {\it fractional Kelvin transform}, in the case $p=2$ we can ``transplant'' supersolutions on $\mathbb{H}^1_+$ to construct suitable supersolutions in a bounded interval. We refer for example to \cite{BZ} and \cite[Appendix A]{ROS} for the definition and properties of the fractional Kelvin transform, in connection with the fractional Laplacian.
\par
In what follows, we use the notation
\[
I=(0,1)\qquad \mbox{ and }\qquad d_I(t)=\min\{t,1-t\},\ \mbox{ for }t\in I.
\]
\begin{lm}
\label{lm:dyda}
Let $0<s<1$ and let $-1<\beta<2\,s$. We consider the function defined by
\[
f_\beta(t):=t^{2\,s-1-\beta}\,(1-t)^\beta,\qquad \mbox{ for } t\in I,
\]
extended by $0$ to the complement of $I$. Then this is a positive local weak solution of the equation
\[
(-\Delta)^s u=\lambda(\beta)\,\frac{u}{\big(t\,(1-t)\big)^{2\,s}},\qquad \mbox{ in }  I.
\]
In particular, it is a positive local weak supersolution of the equation \eqref{equazione}, with $\lambda=\lambda(\beta)$.
\end{lm}
\begin{proof}
We first notice that the last part of the statement easily follows from the fact that 
\[
t\,(1-t)\le \min\{t,1-t\},\qquad \mbox{ for } t\in I.
\]
Let us focus on proving that $f_\beta$ is a solution of the claimed equation.
By still using the notation of the previous subsection, we see that 
\[
f_\beta(t)=t^{2\,s-1}\,U_\beta\left(\frac{1}{t}-1\right),\qquad\mbox{ for }t\in I.
\]
Thus, $f_\beta$ coincides with the {\it fractional Kelvin transform} of the ``shifted'' function 
\[
x\mapsto U_\beta(x-1),
\] 
defined on the half-line $(1,+\infty)$ and extended by $0$ to its complement. Then the proof of the statement above consists in computing the fractional Laplacian of such a Kelvin transform. For every $\varphi\in C^\infty_0((0,1))$, we write
\[
\begin{split}
\iint_{\mathbb{R}\times\mathbb{R}}& \frac{\big(f_\beta(t)-f_\beta(\tau)\big)\,\big(\varphi(t)-\varphi(\tau)\big)}{|t-\tau|^{1+2\,s}}\,dt\,d\tau\\
&=\iint_{\mathbb{R}\times\mathbb{R}} \frac{\left(|t|^{2\,s-1}\,U_\beta\left(\dfrac{1}{t}-1\right)-|\tau|^{2\,s-1}\,U_\beta\left(\dfrac{1}{\tau}-1\right)\right)\,\big(\varphi(t)-\varphi(\tau)\big)}{|t-\tau|^{1+2\,s}}\,dt\,d\tau.
\end{split}
\]
We then make the change of variable $t=1/x$ and $\tau=1/y$, so to get
\[
\begin{split}
\iint_{\mathbb{R}\times\mathbb{R}}& \frac{\big(f_\beta(t)-f_\beta(\tau)\big)\,\big(\varphi(t)-\varphi(\tau)\big)}{|t-\tau|^{1+2\,s}}\,dt\,d\tau\\
&=\iint_{\mathbb{R}\times\mathbb{R}} \frac{\left(|x|^{1-2\,s}\,U_\beta\left(x-1\right)-|y|^{1-2\,s}\,U_\beta\left(y-1\right)\right)\,\left(\varphi\left(\dfrac{1}{x}\right)-\varphi\left(\dfrac{1}{y}\right)\right)}{|x-y|^{1+2\,s}}\,|x|^{2\,s-1}\,|y|^{2\,s-1}\,dx\,dy\\
&=\iint_{\mathbb{R}\times\mathbb{R}} \frac{\left(|y|^{2\,s-1}\,U_\beta\left(x-1\right)-|x|^{2\,s-1}\,U_\beta\left(y-1\right)\right)\,\left(\varphi\left(\dfrac{1}{x}\right)-\varphi\left(\dfrac{1}{y}\right)\right)}{|x-y|^{1+2\,s}}\,dx\,dy.
\end{split}
\]
We now observe that we have the following pointwise identity
\[
\begin{split}
\big(|y|^{2\,s-1}\,U_\beta\left(x-1\right)&-|x|^{2\,s-1}\,U_\beta\left(y-1\right)\big)\,\left(\varphi\left(\dfrac{1}{x}\right)-\varphi\left(\dfrac{1}{y}\right)\right)\\
&=\left(U_\beta\left(x-1\right)-U_\beta\left(y-1\right)\right)\,\left(|x|^{2\,s-1}\,\varphi\left(\dfrac{1}{x}\right)-|y|^{2\,s-1}\,\varphi\left(\dfrac{1}{y}\right)\right)\\
&-\left(U_\beta\left(x-1\right)\,\varphi\left(\frac{1}{x}\right)-U_\beta\left(y-1\right)\,\varphi\left(\frac{1}{y}\right)\right)\, \left(|x|^{2\,s-1}-|y|^{2\,s-1}\right).
\end{split}
\]
This implies that 
\[
\begin{split}
\iint_{\mathbb{R}\times\mathbb{R}}& \frac{\big(f_\beta(t)-f_\beta(\tau)\big)\,\big(\varphi(t)-\varphi(\tau)\big)}{|t-\tau|^{1+2\,s}}\,dt\,d\tau\\
&=\iint_{\mathbb{R}\times\mathbb{R}} \frac{\left(U_\beta\left(x-1\right)-U_\beta\left(y-1\right)\right)\,\left(|x|^{2\,s-1}\,\varphi\left(\dfrac{1}{x}\right)-|y|^{2\,s-1}\,\varphi\left(\dfrac{1}{y}\right)\right)}{|x-y|^{1+2\,s}}\,dx\,dy\\
&-\iint_{\mathbb{R}\times\mathbb{R}} \frac{\left(U_\beta\left(x-1\right)\,\varphi\left(\dfrac{1}{x}\right)-U_\beta\left(y-1\right)\,\varphi\left(\dfrac{1}{y}\right)\right)\, \left(|x|^{2\,s-1}-|y|^{2\,s-1}\right)}{|x-y|^{1+2\,s}}\,dx\,dy.
\end{split}
\]
The second integral on the right-hand side vanishes, thanks to the fact that the function $x\mapsto |x|^{2\,s-1}$ is a local weak solution of 
\[
(-\Delta)^s u=0,\qquad \mbox{ in } \mathbb{R}\setminus \{0\},
\]
(see \cite[Theorem A.4]{BMS}),
once we observe that 
\[
x\mapsto U_\beta\left(x-1\right)\,\varphi\left(\dfrac{1}{x}\right),
\]
is an element of $C^\infty_0((1,+\infty))$.
\par
We can now use the equation solved by $x\mapsto U_\beta(x-1)$: indeed, by Proposition \ref{prop:ACF} for every $\psi\in C^\infty_0((1,+\infty))$ we have 
\[
\iint_{\mathbb{R}\times\mathbb{R}} \frac{\left(U_\beta(x-1)-U_\beta(y-1)\right)\,\big(\psi(x)-\psi(y)\big)}{|x-y|^{1+2\,s}}\,dx\,dy=\lambda(\beta)\,\int_1^{+\infty} \frac{U_\beta(x-1)}{(x-1)^{2\,s}}\,\psi(x)\,dx.
\]
In particular, by choosing 
\[
\psi(x)=|x|^{2\,s-1}\,\varphi\left(\dfrac{1}{x}\right),
\]
and observing that this belongs to $C^\infty_0((1,+\infty))$ if $\varphi\in C^\infty_0((0,1))$, we get
\[
\begin{split}
\iint_{\mathbb{R}\times\mathbb{R}} &\frac{\left(U_\beta\left(x-1\right)-U_\beta\left(y-1\right)\right)\,\left(|x|^{2\,s-1}\,\varphi\left(\dfrac{1}{x}\right)-|y|^{2\,s-1}\,\varphi\left(\dfrac{1}{y}\right)\right)}{|x-y|^{1+2\,s}}\,dx\,dy\\
&=\lambda(\beta)\,\int_1^{+\infty} \frac{U_\beta(x-1)}{(x-1)^{2\,s}}\,|x|^{2\,s-1}\,\varphi\left(\dfrac{1}{x}\right)\,dx.
\end{split}
\]
Finally, by changing back variable $x=1/t$ in the last integral, we get with simple manipulations
\[
\begin{split}
\lambda(\beta)\,\int_1^{+\infty} \frac{U_\beta(x-1)}{(x-1)^{2\,s}}\,|x|^{2\,s-1}\,\varphi\left(\dfrac{1}{x}\right)\,dx
&=\lambda(\beta)\,\int_0^1 \frac{f_\beta(t)}{t^{2\,s}\,(1-t)^{2\,s}}\,\varphi(t)\,dt.
\end{split}
\]
This concludes the proof.
\end{proof}
\begin{oss}
The previous result has been greatly inspired to us by the reading of \cite{Dy}. More precisely, in \cite[Lemma 2.1]{Dy} it is computed the fractional Laplacian of order $s$ of the function
\begin{equation}
\label{dydaw}
w(x)=(1-x^2)^\beta,\qquad \mbox{ for } x\in(-1,1).
\end{equation}
In \cite{Dy} the equation obtained is similar, though a bit different: the computation
uses the Kelvin transformation, as well, even if in a slightly implicit fashion. In other words, in \cite{Dy} the function $w$ is not displayed as the {\it conformal transplantation} of a solution on the half-line: in this respect, we believe that our proof above has its own interest.
\par
In order to compare our function $f_\beta$ with Dyda's one \eqref{dydaw}, we observe that by making the change of variable $x=2\,t-1$, we get
\[
W(t)=w(2\,t-1)=(1-(2\,t-1)^2)^\beta=4^{\beta}\,t^\beta\,(1-t)^\beta,\qquad \mbox{ for } t\in I.
\]
Up to the unessential multiplicative factor $4^\beta$, we see that Dyda's function coincides with ours if and only if 
\[
2\,s-1-\beta=\beta\qquad \mbox{ i.\,e. }\qquad\beta=\frac{2\,s-1}{2}.
\]
Incidentally, we notice that this is the value of $\beta$ which makes $\lambda(\beta)$ the largest possible, by Proposition \ref{lm:costantibeta}.
\end{oss}
\begin{oss}
By recalling Remark \ref{oss:exponent}, from Lemma \ref{lm:dyda} we get in particular that, with the choices $\beta=s$ and $\beta=s-1$, the two functions
\[
t\mapsto t^{s-1}\,(1-t)^s\qquad \mbox{ and }\qquad t\mapsto t^{s}\,(1-t)^{s-1},
\]
are locally weakly $s-$harmonic on $I$.
\end{oss}

\section{Construction of supersolutions in convex sets}
\label{sec:5}

In what follows, for an open set $\Omega\subsetneq\mathbb{R}^N$ we will use the shortcut notation
\[
U_\beta:=d_\Omega^\beta,
\]
where this function is extended by $0$ to the complement $\mathbb{R}^N\setminus \Omega$. In particular, in the borderline case $\beta=0$, this has to be intended as the characteristic function of $\Omega$.
\begin{lm}
\label{lm:sobolevbase}
Let $1<p<\infty$ and $0<s<1$. Let $\Omega\subset\mathbb{R}^N$ be an open bounded convex set. For every
\[
-\frac{1}{p-1}<\beta,
\] 
we have 
\[
U_\beta\in W^{s,p}_{\rm loc}(\Omega)\cap L^{p-1}_{s\,p}(\mathbb{R}^N).
\]
If $\Omega$ is unbounded, then this property is still true, provided that we further assume
\[
\beta<\frac{s\,p}{p-1}.
\] 
\end{lm}
\begin{proof}
The fact that $U_\beta\in W^{s,p}_{\rm loc}(\Omega)$ easily follows from its local Lipschitz character. In order to show that $U_\beta\in L^{p-1}_{s\,p}(\mathbb{R}^N)$, if $\Omega$ is bounded it is enough to show that $U_\beta\in L^{p-1}(\Omega)$. 
To prove this, we can confine ourselves to consider $\beta<0$ (otherwise there is nothing to prove). By using the Coarea Formula and indicating by $r_\Omega$ the supremum of $d_\Omega$ over $\Omega$, we have 
\[
\begin{split}
\int_\Omega U_\beta^{p-1}\,dx&=\int_0^{r_\Omega} t^{\beta\,(p-1)}\,\mathcal{H}^{N-1}(\{x\in\Omega\, :\, d_\Omega=t\})\,dt\le \mathcal{H}^{N-1}(\partial\Omega)\,\int_0^{r_\Omega} t^{\beta\,(p-1)}\,dt.
\end{split}
\]
We then observe that the last integral is finite, provided $\beta>-1/(p-1)$. Observe that we used the monotonicity of the surface area of {\it convex sets} with respect to set inclusion (see \cite[Lemma 2.2.2]{BB}), in the last estimate.
\par
If $\Omega$ is an unbounded convex set, not coinciding with $\mathbb{R}^N$, the proof above still shows that $U_\beta\in W^{s,p}_{\rm loc}\cap L^{p-1}_{\rm loc}(\Omega)$. In order to conclude, we need to prove that 
\[
\int_{\mathbb{R}^N} \frac{U_\beta^{p-1}}{(1+|x|)^{N+s\,p}}\,dx<+\infty,\qquad \mbox{ if } \beta<\frac{s\,p}{p-1}.
\]
For $\beta\le 0$ such a property is straightforward. For $\beta>0$, it is sufficient to fix $x_0\in\partial\Omega$ and observe that (recall that $U_\beta$ vanishes outside $\Omega$)
\[
U_\beta(x)\le |x-x_0|^\beta,\qquad \mbox{ for every } x\in\mathbb{R}^N.
\]
We then obtain
\[
\int_{\mathbb{R}^N} \frac{U_\beta^{p-1}}{(1+|x|)^{N+s\,p}}\,dx\le \int_{\mathbb{R}^N} \frac{|x-x_0|^{\beta\,(p-1)}}{(1+|x|)^{N+s\,p}}\,dx.
\]
It is easily seen that the last integral converges if $\beta\,(p-1)<s\,p$.
\end{proof}
For every $k\in\mathbb{N}$ and $\alpha>0$, we recall that we set
\[
\mathcal{I}(k;\alpha)=\int_0^{+\infty} t^k\,(1+t^2)^{-\frac{k+2+\alpha}{2}}\,dt.
\]
Then we observe that for $N\ge 2$ and every $m>0$, by using the $(N-1)-$dimensional spherical coordinates and a change of variable, we have
\begin{equation}
\label{magic}
\int_{\mathbb{R}^{N-1}} \frac{dy'}{(m^2+|x'-y'|^2)^\frac{N+s\,p}{2}}=\frac{(N-1)\,\omega_{N-1}}{m^{1+s\,p}}\,\mathcal{I}(N-2;s\,p).
\end{equation}
In what follows, we still denote by $\lambda(\beta)$ the constant given by \eqref{lambda}, while $C_{N,s\,p}$ is defined in \eqref{costd}.
 We refer to Remark \ref{oss:meglio} below, for a comment about the sharpness of the restriction $\beta\ge 0$.
\begin{teo}
\label{teo:supersoluzioniconv}
Let $1<p<\infty$ and $0<s<1$. Let $\Omega\subsetneq \mathbb{R}^N$ be an open convex set. 
Then: 
\begin{enumerate}
\item if 
\[
0\le \beta<\frac{s\,p}{p-1},
\]
the function $U_\beta$ is a local weak supersolution of \eqref{equazione}, with $\lambda=C_{N,s\,p}\,\lambda(\beta)$;
\vskip.2cm
\item if $\Omega$ is a half-space and 
\[
-\frac{1}{p-1}<\beta<\frac{s\,p}{p-1},
\]
the function $U_\beta$ is a local weak solution of \eqref{equazione}, still with $\lambda=C_{N,s\,p}\,\lambda(\beta)$.
\end{enumerate} 
\end{teo}
\begin{proof}
We will use a simple geometric construction, already exploited in the proof of \cite[Proposition 3.2]{BC}, in conjunction with the formula \eqref{magic}.
We take $x\in\Omega$ and let $\overline{x}\in\partial\Omega$ be a point such that 
\[
d_\Omega(x)=|x-\overline{x}|.
\]
Since $\Omega$ is convex, there exists a supporting hyperplane for it at the point $\overline{x}$. Without loss of generality, we can suppose that such a supporting hyperplane coincides with 
\[
\mathbb{H}^N_+:=\mathbb{R}^{N-1}\times(0,+\infty),
\]
and thus 
\[
x=(x',x_N)\ \mbox{ with } x_N>0,\qquad \overline{x}=(x',0)\qquad \mbox{ and }\qquad d_\Omega(x)=x_N.
\]
\begin{figure}
\includegraphics[scale=.25]{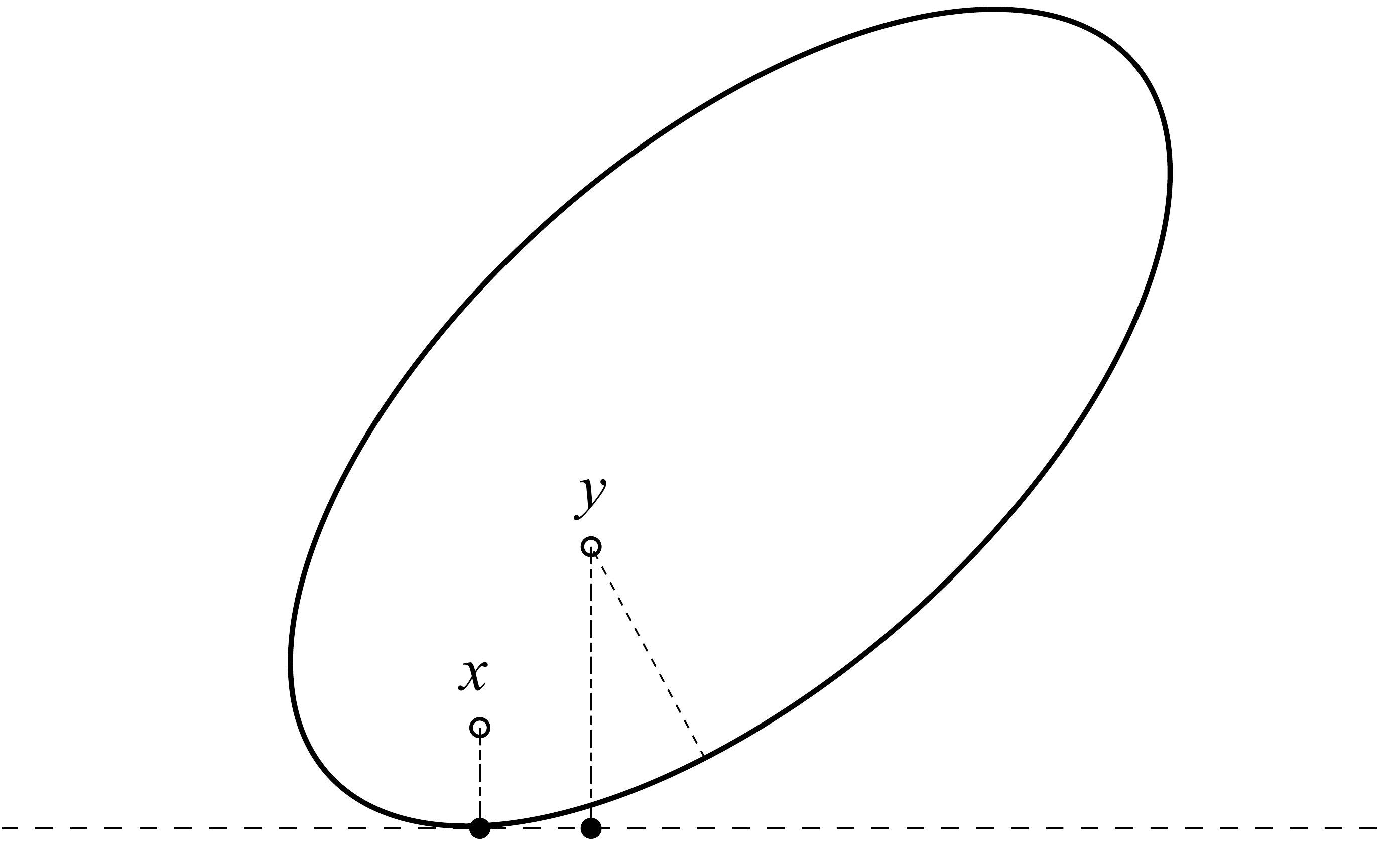}
\caption{The supporting hyperplane for $\Omega$ at $\overline{x}$.}
\label{fig:iperpiano}
\end{figure}
Moreover, we have $\Omega\subset \mathbb{H}_+^N$.
We now observe that for every other $y=(y',y_N)\in\Omega$, by convexity it results 
\[
d_\Omega(y)\le y_N,
\]
see Figure \ref{fig:iperpiano}. By observing that the distance function vanishes in the complement of $\Omega$, we actually have
\[
d_\Omega(y)\le (y_N)_+,\qquad \mbox{ for every } y=(y',y_N)\in\mathbb{R}^N.
\]
By recalling the definition of $U_\beta$, we thus get that for every $y\in \mathbb{R}^N$ and $\beta\ge 0$
\begin{equation}
\label{hop!}
U_\beta(x)-U_\beta(y)=d_\Omega(x)^\beta-d_\Omega(y)^\beta \ge (x_N)_+^\beta-(y_N)_+^\beta.
\end{equation}
For every $0<\varepsilon\ll 1$ and for every $x\in\mathbb{R}^N$, we introduce the conical set
\[
\mathcal{K}_\varepsilon(x)=\Big\{y\in\mathbb{R}^N\, :\, \min\{(1-\varepsilon)\,x_N, (1+\varepsilon)\,x_N \} \le y_N \le \max\{(1-\varepsilon)\,x_N, (1+\varepsilon)\,x_N \} \Big\}.
\]
Recalling that $x_N>0$, we now use \eqref{hop!} and the monotonicity of $\tau\mapsto J_p(\tau)$: we obtain for $\beta\ge 0$ and $0<\varepsilon\ll 1$
\begin{equation}
\label{intermedio}
\begin{split}
\int_{\mathbb{R}^N\setminus \mathcal{K}_\varepsilon(x)} \frac{J_p(U_\beta(x)-U_\beta(y))}{|x-y|^{N+s\,p}}\,dy&\ge \int_{\mathbb{R}^N\setminus \mathcal{K}_\varepsilon(x)} \frac{J_p\left((x_N)_+^\beta-(y_N)_+^\beta\right)}{|x-y|^{N+s\,p}}\,dy.
\end{split}
\end{equation}
If $N\ge 2$, the last integral can be written as
\[
\begin{split}
\int_{\mathbb{R}^N\setminus \mathcal{K}_\varepsilon(x)} \frac{J_p(U_\beta(x)-U_\beta(y))}{|x-y|^{N+s\,p}}\,dy&\ge \int_{\mathbb{R}^N\setminus \mathcal{K}_\varepsilon(x)} \frac{J_p\left((x_N)_+^\beta-(y_N)_+^\beta\right)}{|x-y|^{N+s\,p}}\,dy \\
&=\int_{\mathbb{R}\setminus I_\varepsilon(x_N)} J_p\left((x_N)_+^\beta-(y_N)_+^\beta\right)\\
&\times \left(\int_{\mathbb{R}^{N-1}} \frac{dy'}{(|x_N-y_N|^2+|x'-y'|^2)^\frac{N+s\,p}{2}}\right)\,dy_N,\\
\end{split}
\]
where $I_\varepsilon(x_N)$ is the same interval as in \eqref{intervallino}.
If we now use \eqref{magic} with $m=|x_N-y_N|$, we get 
\[
\int_{\mathbb{R}^{N-1}} \frac{dy'}{(|x_N-y_N|^2+|x'-y'|^2)^\frac{N+s\,p}{2}}=\frac{C_{N,s\,p}}{|x_N-y_N|^{1+s\,p}}.
\]
Thus, we obtain from \eqref{intermedio}
\[
\begin{split}
\int_{\mathbb{R}^N\setminus \mathcal{K}_\varepsilon(x)} \frac{J_p(U_\beta(x)-U_\beta(y))}{|x-y|^{N+s\,p}}\,dy\ge C_{N,s\,p}\,\int_{\mathbb{R}\setminus I_\varepsilon(x_N)} \frac{J_p\left((x_N)_+^\beta-(y_N)_+^\beta\right)}{|x_N-y_N|^{1+s\,p}}\,dy_N.
\end{split}
\]
By recalling that we set $C_{1,s,p}=1$, the above formula obviously holds for $N=1$, as well: actually, it coincides with \eqref{intermedio}.
\par
By the definition \eqref{Fe} and the identity \eqref{beh!}, we have for $x_N>0$
\[
2\,\int_{\mathbb{R}\setminus I_\varepsilon(x_N)} \frac{J_p\left((x_N)_+^\beta-(y_N)_+^\beta\right)}{|x_N-y_N|^{1+s\,p}}\,dy_N=
F_\varepsilon(x_N)=\lambda_\varepsilon(\beta)\,\frac{x_N^{\beta(p-1)}}{x_N^{s\,p}}=\lambda_\varepsilon(\beta)\,\frac{U_\beta(x)^{p-1}}{d_\Omega(x)^{s\,p}}.
\]
Moreover, we recall that (see the proof of Proposition \ref{prop:ACF})
\[
\begin{split}
\lambda_\varepsilon(\beta)&=2\,\int_0^{1-\varepsilon}\frac{J_p(1-\tau^\beta)}{|1-\tau|^{1+s\,p}}\,(1-\tau^{s\,p-\beta\,(p-1)})\,d\tau \\
&+2\,\int_{1-\varepsilon}^\frac{1}{1+\varepsilon}\frac{J_p(\tau^\beta-1)}{|\tau-1|^{1+s\,p}}\,\tau^{s\,p-1-\beta\,(p-1)}d\tau+\frac{2}{s\,p},
\end{split}
\]
and
\[
\lim_{\varepsilon\to 0^+} \lambda_\varepsilon(\beta)=\lambda(\beta).
\]
Thus we have 
\[
2\,C_{N,s\,p}\,\int_{\mathbb{R}\setminus I_\varepsilon(x_N)} \frac{J_p\left((x_N)_+^\beta-(y_N)_+^\beta\right)}{|x_N-y_N|^{1+s\,p}}\,dy_N=C_{N,s\,p}\,\lambda_\varepsilon(\beta)\,\frac{U_{\beta}(x)^{p-1}}{d_\Omega(x)^{s\,p}}.
\]
This in turn leads to
\[
2\,\int_{\mathbb{R}^N\setminus \mathcal{K}_\varepsilon(x)} \frac{J_p(U_\beta(x)-U_\beta(y))}{|x-y|^{N+s\,p}}\,dy\ge C_{N,s\,p}\,\lambda_\varepsilon(\beta)\,\frac{U_\beta(x)^{p-1}}{d_\Omega(x)^{s\,p}}.
\]
We take $\varphi\in C^\infty_0(\Omega)$ non-negative, multiply the previous inequality by $\varphi(x)$ and integrate over $\Omega$. We get 
\begin{equation}
\label{stanco1}
\begin{split}
2\,\int_\Omega \left(\int_{\mathbb{R}^N\setminus \mathcal{K}_\varepsilon(x)} \frac{J_p(U_\beta(x)-U_\beta(y))}{|x-y|^{N+s\,p}}\,dy\right)\,\varphi(x)\,dx&\ge C_{N,s\,p}\,\lambda_\varepsilon(\beta)\,\int_\Omega \frac{U_\beta(x)^{p-1}}{d_\Omega(x)^{s\,p}}\,\varphi(x)\,dx.
\end{split}
\end{equation}
On the other hand, we have
\[
\frac{J_p(U_\beta(x)-U_\beta(y))\,(\varphi(x)-\varphi(y))}{|x-y|^{N+s\,p}}\in L^1(\mathbb{R}^N\times\mathbb{R}^N),
\]
thanks to Lemma \ref{lm:sobolevbase}. Thus by the Dominated Convergence Theorem, we get
\[
\begin{split}
\iint_{\mathbb{R}^N\times\mathbb{R}^N}&\frac{J_p(U_\beta(x)-U_\beta(y))\,(\varphi(x)-\varphi(y))}{|x-y|^{N+s\,p}}\,dx\,dy\\
&=\lim_{\varepsilon\to 0^+} \iint_{(\mathbb{R}^N\times\mathbb{R}^N)\setminus \mathcal{C}_\varepsilon}\frac{J_p(U_\beta(x)-U_\beta(y))\,(\varphi(x)-\varphi(y))}{|x-y|^{N+s\,p}}\,dx\,dy,\\
\end{split}
\]
where
\[
\mathcal{C}_\varepsilon=\Big\{(x,y)\in\mathbb{R}^N\times\mathbb{R}^N\, :\, \min\{(1-\varepsilon)\,x_N, (1+\varepsilon)\,x_N \} \le y_N \le \max\{(1-\varepsilon)\,x_N, (1+\varepsilon)\,x_N \} \Big\}.
\]
Moreover, for every $0<\varepsilon\ll 1$ we have
\[
\frac{J_p(U_\beta(x)-U_\beta(y))}{|x-y|^{N+s\,p}}\,\varphi(x)\in L^1((\mathbb{R}^N\times\mathbb{R}^N)\setminus \mathcal{C}_\varepsilon).
\]
Thus with a simple change of variables, we get
\begin{equation}
\label{stanco2}
\begin{split}
\iint_{\mathbb{R}^N\times\mathbb{R}^N}&\frac{J_p(U_\beta(x)-U_\beta(y))\,(\varphi(x)-\varphi(y))}{|x-y|^{N+s\,p}}\,dx\,dy\\
&=\lim_{\varepsilon\to 0^+} \iint_{(\mathbb{R}^N\times\mathbb{R}^N)\setminus \mathcal{C}_\varepsilon}\frac{J_p(U_\beta(x)-U_\beta(y))\,(\varphi(x)-\varphi(y))}{|x-y|^{N+s\,p}}\,dx\,dy.\\
&=2\,\lim_{\varepsilon\to 0^+} \int_{\Omega} \left(\int_{\mathbb{R}^N\setminus \mathcal{K}_\varepsilon(x)}\frac{J_p(U_\beta(x)-U_\beta(y))}{|x-y|^{N+s\,p}}\,dy\right)\,\varphi(x)\,dx.
\end{split}
\end{equation}
Observe that we also used Fubini's Theorem for every fixed $0<\varepsilon\ll 1$, in order to arrive at the last integral. By joining \eqref{stanco1} and \eqref{stanco2}, we finally get
\[
\begin{split}
\iint_{\mathbb{R}^N\times\mathbb{R}^N}&\frac{J_p(U_\beta(x)-U_\beta(y))\,(\varphi(x)-\varphi(y))}{|x-y|^{N+s\,p}}\,dx\,dy\ge C_{N,s\,p}\,\lambda(\beta)\,\int_\Omega \frac{U_\beta(x)^{p-1}}{d_\Omega(x)^{s\,p}}\,\varphi(x)\,dx,
\end{split}
\]
which is the desired conclusion for $\beta\ge 0$. 
\par In order to prove the second statement, we first observe that if $\Omega$ is a half-space, we can assume for simplicity that
\[
\Omega=\mathbb{H}^N_+.
\] 
Then, in the case $N=1$, the statement has been proved in Proposition \ref{prop:ACF}. For $N\ge 2$, it suffices to observe that
 we have equalities everywhere in the previous argument, even for $\beta<0$, provided it is an admissible exponent.
\end{proof}

\begin{oss}[Optimality of Theorem \ref{teo:supersoluzioniconv}]
\label{oss:meglio}
As a consequence of Proposition \ref{lm:costantibeta}, we have that 
\[
C_{N,s\,p}\,\lambda(\beta)\le C_{N,s\,p}\,\lambda\left(\frac{s\,p-1}{p}\right),\qquad \mbox{ for every } -\frac{1}{p-1}<\beta<\frac{s\,p}{p-1}.
\]
Thus, even in the more general case of a convex subset $\Omega$, the choice 
\[
\beta=\frac{s\,p-1}{p},
\]
still produces a supersolution of \eqref{equazione}, which has the largest possible $\lambda$, among supersolutions of this type. However, it should be noticed that, in light of Theorem \ref{teo:supersoluzioniconv}, such a choice is now feasible {\it only for} 
\[
\frac{s\,p-1}{p}\ge 0\qquad \mbox{ i.\,e. }\qquad s\,p\ge 1,
\]
unless $\Omega$ is a half-space.
Moreover, if the convex set $\Omega$ is not a half-space, {\it such a result is optimal} in the following sense: already in the borderline case $s\,p=1$, the function $U_\beta$ with $\beta<0$
{\it is not} a supersolution of \eqref{equazione}. See Lemma \ref{lm:bum!} below for a simple counter-example.
\end{oss}

\section{The sharp fractional Hardy inequality for convex sets}
\label{sec:6}

In this section we still use the notation \eqref{halfspaces}. We start with the following general fact.
\begin{prop}
\label{prop:convessiH}
Let $1<p<\infty$ and $0<s<1$. For every $\Omega\subsetneq\mathbb{R}^N$ open convex set, we have 
\[
\mathfrak{h}_{s,p}(\Omega)\le \mathfrak{h}_{s,p}(\mathbb{H}^N_+).
\]
\end{prop}
\begin{proof}
In dimension $N=1$, we suppose $\Omega\subsetneq\mathbb{R}$ to be a bounded interval. Thanks to \eqref{scalina}, we can assume that $\Omega=I=(0,1)$. We take $\psi\in C^\infty_0(\mathbb{H}^1_+)$ and define the rescaled function
\[
\psi_\varepsilon(t)=\psi\left(\frac{t}{\varepsilon}\right).
\]
We observe that for $\varepsilon>0$ sufficiently small, we have $\psi_\varepsilon\in C^\infty_0(I)$. We compute
\[
[\psi_\varepsilon]^p_{W^{s,p}(\mathbb{R})}=\varepsilon^{1-s\,p}\,[\psi]^p_{W^{s,p}(\mathbb{R})},
\]
and 
\[
\int_I \frac{|\psi_\varepsilon|^p}{d_I^{s\,p}}\,dt=\int_0^1 \frac{\left|\psi\left(\dfrac{t}{\varepsilon}\right)\right|^p}{(\min\{t,1-t\})^{s\,p}}\,dt=\varepsilon^{1-s\,p}\,\int_0^\frac{1}{\varepsilon} \frac{|\psi(\tau)|^p}{\left(\min\left\{\tau,\varepsilon^{-1}-\tau\right\}\right)^{s\,p}}\,d\tau.
\]
By recalling that $\psi$ is compactly supported, we get that for $0<\varepsilon\ll 1$ we have
\[
\min\left\{\tau,\varepsilon^{-1}-\tau\right\}=\tau,\qquad \mbox{ for } \tau \mbox{ in the support ot } \psi.
\]
In conclusion, we get for every $\psi\in C^\infty_0(\mathbb{H}^1_+)$
\[
\mathfrak{h}_{s,p}(I)\le\lim_{\varepsilon\to 0^+} \frac{[\psi_\varepsilon]^p_{W^{s,p}(\mathbb{R})}}{\displaystyle\int_I \frac{|\psi_\varepsilon|^p}{d_I^{s\,p}}\,dt}=\lim_{\varepsilon\to 0^+}\frac{[\psi]^p_{W^{s,p}(\mathbb{R})}}{\displaystyle\int_0^\frac{1}{\varepsilon} \frac{|\psi(\tau)|^p}{\tau^{s\,p}}\,d\tau}=\frac{[\psi]^p_{W^{s,p}(\mathbb{R})}}{\displaystyle\int_0^{+\infty} \frac{|\psi(\tau)|^p}{\tau^{s\,p}}\,d\tau}.
\]
By arbitrariness of $\psi$, this gives the claimed inequality. 
\par
For the case $N\ge 2$, we can repeat the same proof of \cite[Theorem 5]{MMP}, which deals with the local case, up to some very minor modifications. We just recall that the proof in \cite{MMP} is based on a scaling argument as in the one-dimensional case exposed above, together with the fact that a convex set admits a tangent hyperplane at almost every boundary point.
\end{proof}

\subsection{The case of the half-space}
For an half-space, we can determine the sharp Hardy constant without restrictions on the product $s\,p$.
\begin{teo}
\label{teo:half}
Let $1<p<\infty$ and $0<s<1$. For every $N\ge 1$  we have
\[
\mathfrak{h}_{s,p}(\mathbb{H}^N_+)=C_{N,s\,p}\,\Lambda_{s,p},
\]
where $\Lambda_{s,p}$ and $C_{N,s\,p}$ are defined by \eqref{cost1d} and \eqref{costd}, respectively.
Moreover, such a constant is not attained.
\end{teo}
\begin{proof}
By combining \eqref{charintro} and Theorem \ref{teo:supersoluzioniconv} for $\Omega=\mathbb{H}^N_+$, we immediately obtain
\[
\mathfrak{h}_{s,p}(\mathbb{H}_+^N)\ge C_{N,s\,p}\,\lambda(\beta),\qquad \mbox{ for every } -\frac{1}{p-1}<\beta<\frac{s\,p}{p-1}.
\]
Moreover, by Proposition \ref{lm:costantibeta}, we know that the right-hand side is maximal for $\beta=(s\,p-1)/p$ and thus 
\[
\mathfrak{h}_{s,p}(\mathbb{H}_+^N)\ge C_{N,s\,p}\,
\lambda\left(\frac{s\,p-1}{p}\right)=C_{N,s\,p}\,
\Lambda_{s,p}.
\]
In order to prove that the right-hand side actually gives the sharp constant, we distinguish two cases: $N=1$ and $N\ge 2$. We will show that the latter reduces to the former: this is quite a standard fact for the Hardy inequality, but we prefer to give the details, since some non-trivial computations are needed. For the case $N=1$, we will use a slightly different family of trial functions with respect to \cite{FS, FSspace}: this permits to treat the cases $s\,p<1$ and $s\,p\ge 1$ at the same time.
\vskip.2cm\noindent
{\it Sharpness: case $N=1$.} We need to prove that
\[
\mathfrak{h}_{s,p}(\mathbb{H}^1_+) \le \Lambda_{s,p}.
\]
We take a cut-off function $\psi\in C^\infty_0\left(\left(-\infty,\,2\right)\right)$ such that
\[
0\le\psi\le1,\qquad\psi\equiv 1, \mbox{ on } \in [0,1],\qquad|\psi'|\le \widetilde{C},
\]
and we use the trial function
\[
\phi_{\beta}=U_\beta\,\psi,\qquad \mbox{ with } \frac{s\,p-1}{p}<\beta<s.
\]
According to Lemma \ref{lm:potenzesobolev} and Lemma \ref{lm:lemmaproduct}, this function belongs to $\widetilde{W}^{s,p}_0(\mathbb{H}^1_+)$. In light of the estimate \eqref{eq:stimaseminorma1} and the properties of the cut-off, we get
\begin{equation}\label{eq:quoziente1d}
\begin{split}
\mathfrak{h}_{s,p}(\mathbb{H}^1_+)&\le\frac{[\phi_\beta]^p_{W^{s,p}(\mathbb{R})}}{\displaystyle\int_0^{+\infty}\dfrac{|\phi_\beta(x)|^p}{x^{s\,p}}\,dx}\le  \frac{[U_\beta\,\psi]^p_{W^{s,p}((0,2))}}{\displaystyle\int_0^{2} \frac{(U_\beta\,\psi)^p}{x^{s\,p}}\,dx}+\frac{2}{s\,p}+\frac{2^{1+p-s\,p}}{s\,p}\,\widetilde{C}^p\,\frac{\|U_\beta\|^p_{L^p((0,2))}}{\displaystyle\int_0^{2} \frac{(U_\beta\,\psi)^p}{x^{s\,p}}\,dx}.
\end{split}
\end{equation}
We evaluate separately the two quotients on the right-hand side. For the first one, by using the estimate \eqref{eq:stimaseminorma2}, we have
\[
\begin{split}
\frac{[U_\beta\,\psi]_{W^{s,p}((0,2))}}{\displaystyle\left(\int_0^{2} \frac{(U_\beta\,\psi)^p}{x^{s\,p}}\,dx\right)^\frac{1}{p}}\le  \frac{[U_\beta]_{W^{s,p}((0,2))}}{\displaystyle\left(\int_0^{2} \frac{(U_\beta\,\psi)^p}{x^{s\,p}}\,dx\right)^\frac{1}{p}}+\left(\frac{C}{s\,(1-s)}\right)^{\frac{1}{p}}\,\widetilde{C}^s\,\frac{\|U_\beta\|_{L^p((0,2))}}{\displaystyle\left(\int_0^{2} \frac{(U_\beta\,\psi)^p}{x^{s\,p}}\,dx\right)^\frac{1}{p}}.
\end{split}
\]
Moreover, by recalling the definition of $U_\beta$, we note that
\[ 
\lim_{\beta\to \left(\frac{s\,p-1}{p}\right)^+}\int_0^{2} \frac{\left(U_\beta\,\psi\right)^p}{x^{s\,p}}\,dx =
\lim_{\beta\to \left(\frac{s\,p-1}{p}\right)^+}\int_0^{2}\dfrac{\psi^{p}}{x^{s\,p-\beta\,p}}\,dx=+\infty,
 \]
thus the denominator diverges as $\beta$ goes to $(s\,p-1)/p$. Coming back to \eqref{eq:quoziente1d}, this entails that
\[
\begin{split}
\mathfrak{h}_{s,p}(\mathbb{H}^1_+)&\le \limsup_{\beta\to \left(\frac{s\,p-1}{p}\right)^+}\left[\frac{[U_\beta\,\psi]^p_{W^{s,p}((0,2))}}{\displaystyle\int_0^{2} \frac{(U_\beta\,\psi)^p}{x^{s\,p}}\,dx}+\frac{2}{s\,p}+\frac{2^{p+1-s\,p}}{s\,p}\,\widetilde{C}^p\,\frac{\|U_\beta\|^p_{L^p((0,2))}}{\displaystyle\int_0^{2} \frac{(U_\beta\,\psi)^p}{x^{s\,p}}\,dx}\right]\\
&\le \limsup_{\beta\to \left(\frac{s\,p-1}{p}\right)^+}\frac{[U_\beta]^p_{W^{s,p}((0,2))}}{\displaystyle\int_0^{2} \frac{(U_\beta\,\psi)^p}{x^{s\,p}}\,dx}+\frac{2}{s\,p}\le \limsup_{\beta\to \left(\frac{s\,p-1}{p}\right)^+}\frac{[U_\beta]^p_{W^{s,p}((0,2))}}{\displaystyle\int_0^{1} \frac{(U_\beta)^p}{x^{s\,p}}\,dx}+\frac{2}{s\,p}.
\end{split}
\]
We claim that 
\begin{equation}
\label{manca}
\limsup_{\beta\to \left(\frac{s\,p-1}{p}\right)^+} \frac{[U_\beta]^p_{W^{s,p}((0,2))}}{\displaystyle\int_0^1 \frac{(U_\beta)^p}{x^{s\,p}}\,dx}\le 2\,\int_0^1 \dfrac{\left|1-t^\frac{s\,p-1}{p}\right|^p}{(1-t)^{1+s\,p}}\,dt,
\end{equation}
this would conclude the proof, by recalling the definition \eqref{cost1d} of $\Lambda_{s,p}$.
By using the form of $U_\beta$ we have 
\[
\int_0^1 \frac{(U_\beta)^p}{x^{s\,p}}\,dx=\int_0^1 x^{\beta\,p-s\,p}\,dt=\frac{1}{\beta\,p-s\,p+1}.
\]
Thus in order to prove \eqref{manca}, we just need to show that 
\[
\limsup_{\beta\to \left(\frac{s\,p-1}{p}\right)^+}(\beta\,p-s\,p+1)\,[U_\beta]^p_{W^{s,p}(0,2)}\le 2\,\int_0^1 \dfrac{\left|1-t^\frac{s\,p-1}{p}\right|^p}{(1-t)^{1+s\,p}}\,dt.
\]
By recalling the estimate \eqref{stima1dsharp} from Remark \ref{oss:miserve!}, we have 
\[
\begin{split}
[U_\beta]^p_{W^{s,p}((0,2))}\le \left(\int_0^1\dfrac{|1-t^\beta|^p}{|1-t|^{1+s\,p}}\,\left(1+t^{s\,p-p\,\beta-1}\right)\,dt\right)\,\frac{2^{\beta\,p-s\,p+1}}{\beta\,p-s\,p+1}.
\end{split}
\]
Hence, by taking the limit as $\beta$ goes to $(s\,p-1)/p$ and using the Dominated Convergence Theorem, we get \eqref{manca}, as desired. This proves the sharpness for $N=1$.
\vskip.2cm\noindent
{\it Sharpness: case $N\ge 2$.} We will show that this can reduced to the previous case, 
by proceeding as in \cite[Theorem 1.1]{FS} and \cite[Proposition 3.2]{MSK}. Let $\eta\in C^\infty_0((0, +\infty))$ and $\chi\in C^\infty_0((-1,1)^{N-1})$, we use the test function
\[
\varphi=\chi_M(x')\, \eta(x_N),\qquad \mbox{ where } \chi_M(x'):=\frac{M^{\frac{1-N}{p}}}{\|\chi\|_{L^p(\mathbb{R}^{N-1})}}\, \chi\left(\frac{x'}{M}\right).
\]
Observe that by construction the function $\chi_M$ has compact support on $(-M,M)^{N-1}$ and unit $L^p$ norm. We thus obtain
	\[ 
	\begin{split}
	\mathfrak{h}_{s,p}(\mathbb{H}^N_+) \le \dfrac{\left[ \eta\, \chi_M \right]^p_{W^{s,p}(\mathbb{R}^N)}}{\displaystyle\int_{\mathbb{H}^N_+} \dfrac{\left(\eta\, \chi_M\right)^p}{x_N^{s\,p}} \, dx}=\dfrac{\left[ \eta \, \chi_M \right]^p_{W^{s,p}(\mathbb{R}^N)}}{\displaystyle\int_{0}^{+\infty} \dfrac{\eta^p}{x_N^{s\,p}} \, dx_N},
	\end{split}
	\]
	where in the last identity we used Fubini's Theorem and the properties of $\chi_M$. 
	In order to estimate the seminorm, we first use Minkowski's inequality
\[
	\begin{split}
	\big[ \eta \, \chi_M \big]_{W^{s,p}(\mathbb{R}^N)} \le \left( \iint_{\mathbb{R}^N\times\mathbb{R}^N}\dfrac{|\chi_M(x')|^p \,|\eta(x_N)-\eta(y_N)|^p}{|x-y|^{N+s\,p}}\,dx\,dy \right)^{\frac{1}{p}}\\
	 + \left( \iint_{\mathbb{R}^N\times\mathbb{R}^N}\dfrac{|\eta(y_N)|^p \, |\chi_M(x')-\chi_M(y')|^p}{|x-y|^{N+s\,p}}\,dx\,dy \right)^{\frac{1}{p}},
	 \end{split}
\]
	and we focus separately on the two integrals on the right-hand side. For the first integral, we use Fubini's Theorem and the identity \eqref{magic}, so to get
	\[
	\begin{split}
	\iint_{\mathbb{R}^N\times\mathbb{R}^N}&\dfrac{|\chi_M(x')|^p \,|\eta(x_N)-\eta(y_N)|^p}{|x-y|^{N+s\,p}}\,dx\,dy \\
	&= C_{N,s\,p}\,\left( \int_{\mathbb{R}^{N-1}} |\chi_M(x')|^p \, dx'\right)\,\left(\iint_{\mathbb{R}\times \mathbb{R}} \frac{|\eta(x_N)-\eta(y_N)|^p}{|x_N-y_N|^{1+s\,p}}dx_N\,dy_N\right) =C_{N,s\,p} \,[\eta]^p_{W^{s,p}(\mathbb{R})}.
	\end{split}
	\]
	On the other hand, by using a computation similar to \eqref{magic}, we have
\[ \int_{\mathbb{R}} \frac{1}{\left( |x_N-y_N|^2 + |x'-y'|^2 \right)^{\frac{N+s\,p}{2}}} \, dx_N = \frac{2\,\displaystyle\int_0^{+\infty} (1+t^2)^{-\frac{N+s\,p}{2}}\,dt}{|x'-y'|^{N-1+s\,p}}=\frac{C}{|x'-y'|^{N-1+s\,p}}.
	\]
	Thus, it holds 
	\[
	\begin{split}
	\iint_{\mathbb{R}^N\times\mathbb{R}^N}&\dfrac{|\eta(y_N)|^p \, |\chi_M(x')-\chi_M(y')|^p}{|x-y|^{N+s\,p}}\,dx\,dy \\
	&= C \int_{\mathbb{R}} |\eta(y_N)|^p \, dy_N\,\left( \iint_{\mathbb{R}^{N-1}\times\mathbb{R}^{N-1}} \frac{|\chi_M(x')-\chi_M(y')|^p}{|x'-y'|^{N-1+s\,p}}\,dx'\,dy'\right)\\
	&= C\, \frac{\| \eta \|^p_{L^p(\mathbb{R})}}{\|\chi\|^p_{L^p(\mathbb{R}^{N-1})}}\, \frac{[\chi ]^p_{W^{s,p}(\mathbb{R}^{N-1})}}{M^{s\,p}}.
	\end{split}
	\]
	In the last identity we used the definition of $\chi_M$ and a change of variable.
	Then, it follows that
	\[
	\begin{split}
	\Big(\mathfrak{h}_{s,p}(\mathbb{H}^N_+)\Big)^\frac{1}{p} &\le \left(C_{N,s\,p}\right)^\frac{1}{p}\,\dfrac{[\eta]_{W^{s,p}(\mathbb{R})}}{\displaystyle \left(\int_{0}^{+\infty}\dfrac{\eta^{p}}{x_N^{s\,p}} \, dx \right)^\frac{1}{p}}+ C^\frac{1}{p}\,\frac{\| \eta \|_{L^p(\mathbb{R})}}{M^s\,\|\chi\|_{L^p(\mathbb{R}^{N-1})}}\, \frac{[ \chi]_{W^{s,p}(\mathbb{R}^{N-1})}}{\displaystyle \left(\int_{0}^{+\infty}\dfrac{\eta^{p}}{x_N^{s\,p}} \, dx \right)^\frac{1}{p}} \\
	\end{split}
	\]
	By letting $M$ go to $+\infty$ and thanks to the arbitrariness of $\eta \in C^{\infty}_0((0, +\infty))$, we obtain
	\[ 
	\mathfrak{h}_{s,p}(\mathbb{H}^N_+) \le C_{N,s\,p} \; \mathfrak{h}_{s,p}(\mathbb{H}^1_+)=C_{N,s\,p}\,\Lambda_{s,p},
	\]
as desired. The last identity follows from the sharpness for $N=1$.
\vskip.2cm\noindent
The fact that $\mathfrak{h}_{s,p}(\mathbb{H}^N_+)$ is not attained follows directly from Proposition \ref{prop:nonne}, since by Theorem \ref{teo:supersoluzioniconv} we found a local weak solution of \eqref{equazione} with $\lambda=\mathfrak{h}_{s,p}(\mathbb{H}^N_+)$, of the form
\[
u=d_{\mathbb{H}^N_+}^\frac{s\,p-1}{p}.
\]
The proof is over.

\end{proof}

\subsection{The case $s\,p\ge 1$}
\begin{teo}
\label{teo:sp}
Let $1<p<\infty$ and $0<s<1$ be such that $s\,p\ge 1$. Then for every $\Omega\subsetneq\mathbb{R}^N$ open convex set, we have
\[
\mathfrak{h}_{s,p}(\Omega)=C_{N,s\,p}\,\Lambda_{s,p},
\]
where $\Lambda_{s,p}$ and $C_{N,s\,p}$ are defined by \eqref{cost1d} and \eqref{costd}, respectively.
Moreover, such a constant is not attained.
\end{teo}
\begin{proof}
We suppose that $\Omega$ is not a half-space, otherwise there is nothing to prove. By appealing to \eqref{charintro} and Theorem \ref{teo:supersoluzioniconv}, we get  
\begin{equation}
\label{bazo}
\mathfrak{h}_{s,p}(\Omega)\ge C_{N,s\,p}\,\lambda(\beta),\qquad \mbox{ for every } 0\le \beta<\frac{s\,p}{p-1}.
\end{equation}
Again by Proposition \ref{lm:costantibeta}, we know that the right-hand side is maximal for $\beta=(s\,p-1)/p$: such a choice is feasible, thanks to the assumption $s\,p\ge 1$. We thus get
\[
\mathfrak{h}_{s,p}(\Omega)\ge C_{N,s\,p}\,\lambda\left(\frac{s\,p-1}{p}\right)=C_{N,s\,p}\,\Lambda_{s,p}.
\]
On the other hand, by Proposition \ref{prop:convessiH} we know that
\[
\mathfrak{h}_{s,p}(\Omega) \le \mathfrak{h}_{s,p}(\mathbb{H}_+^N).
\]
By combining the latter with Theorem \ref{teo:half}, we finally get
\[
\mathfrak{h}_{s,p}(\Omega)\le C_{N,s\,p}\,\Lambda_{s,p},
\]
as well. This proves that $\mathfrak{h}_{s,p}(\Omega)$ has the claimed expression.
\par
Finally, the fact that $\mathfrak{h}_{s,p}(\Omega)$ is not attained follows directly from Proposition \ref{prop:nonne}, since by Theorem \ref{teo:supersoluzioniconv} we found a local weak supersolution of \eqref{equazione} with $\lambda=\mathfrak{h}_{s,p}(\Omega)$, having the form
\[
u=d_\Omega^\frac{s\,p-1}{p}.
\]
This concludes the proof.
\end{proof}

\begin{oss}[A lower bound in the case $s\,p<1$]
\label{oss:forbice}
As already said, in the case $s\,p<1$ the maximal choice for $\beta$ is not feasbile. In this case we can choose in \eqref{bazo} the exponent $\beta=0$ and get at least a lower bound, i.e.
\[
\mathfrak{h}_{s,p}(\Omega)\ge C_{N,s\,p}\,\lambda(0)=C_{N,s\,p}\,\frac{2}{s\,p}.
\] 
\end{oss}
\subsection{The case $0<s<1$ and $p=2$}
We first highlight the following consequence of Lemma \ref{lm:dyda}. The resulting inequality is the same as \cite[Corollary 1.3]{Dy} by Bart\l omiej Dyda.
\begin{prop}
\label{prop:dyda}
Let $0<s<1$ and let $a<b$ be two real numbers. We have the following one-dimensional Hardy-type inequality
\begin{equation}
\label{dydone}
\Lambda_{s,2}\,\int_a^b |u|^2\,\left[\frac{1}{t-a}+\frac{1}{b-t}\right]^{2\,s}\,dt\le [u]_{W^{s,2}(\mathbb{R})}^2,
\end{equation}
for every $u\in \widetilde{W}^{s,2}_0((a,b))$.
In particular, we have 
\[
\mathfrak{h}_{s,2}((a,b))=\Lambda_{s,2}.
\]
and such a constant is not attained.
\end{prop}
\begin{proof}
By recalling \eqref{scalina}, we can suppose that $(a,b)=(0,1)=I$.
Then the proof of the inequality \eqref{dydone} is the same as in \cite{Dy} and it follows the same lines as that of \cite[Lemma 4.1, point (i)]{BBZ}: it is sufficient to take 
\[
f_\beta(t):=t^{2\,s-1-\beta}\,(1-t)^\beta,\qquad \mbox{ with } \beta=\frac{2\,s-1}{2},
\]
observe that this weakly solves 
\[
(-\Delta)^s u=\Lambda_{s,2}\,\frac{u}{\big(t\,(1-t)\big)^{2\,s}},\qquad \mbox{ in }  I,
\]
and then make a suitable application of the discrete Picone inequality.
\par
As for the sharp fractional Hardy constant, by observing that 
\[
\frac{1}{t-a}+\frac{1}{b-t}\ge \frac{1}{\min\{t-a,b-t\}}=\frac{1}{d_I(t)},\qquad \mbox{ for } t\in I,
\]
we easily get from \eqref{dydone} that
\[
\mathfrak{h}_{s,2}((a,b))\ge \Lambda_{s,2}.
\]
The reverse inequality can be obtained from Proposition \ref{prop:convessiH} and Theorem \ref{teo:half}. Finally, the fact that the sharp constant is not attained can be inferred again from Proposition \ref{prop:nonne}.
\end{proof}
By combining the previous one-dimensional result with a decomposition of the Gagliardo-Slobodecki\u{\i} seminorm, taken from \cite{LS} (see also \cite[Chapter 1, Section 5]{Da}), we can finally compute the sharp fractional Hardy constant of a convex set for $p=2$, {\it without restrictions on} $0<s<1$. 
\par
This complements \cite[Theorem 5, points (i) \& (ii)]{FMT}, where the case $0<s<1/2$ was left open. 
\begin{teo}
\label{teo:p2}
Let $0<s<1$, then for every $\Omega\subsetneq\mathbb{R}^N$ open convex set, we have
\[
\mathfrak{h}_{s,2}(\Omega)=C_{N,s,2}\,\Lambda_{s,2},
\]
where $\Lambda_{s,2}$ and $C_{N,s,2}$ are defined by \eqref{cost1d} and \eqref{costd}, respectively.
Moreover, such a constant is not attained.
\end{teo}
\begin{proof}
By joining Proposition \ref{prop:convessiH} and Theorem \ref{teo:half}, we already know that 
\[
\mathfrak{h}_{s,2}(\Omega)\le \mathfrak{h}_{s,2}(\mathbb{H}^N_+)=C_{N,s,2}\,\Lambda_{s,2}.
\]
In order to prove the reverse estimate, it is sufficient to reproduce verbatim the proof of \cite[Theorem 1.1]{LS} for convex sets, by replacing the $W^{s,p}$ seminorm on $\Omega$ there with that on $\mathbb{R}^N$. In particular, we need to use the following
\textit{reduction formula}
\[
\begin{split}
\iint_{\mathbb{R}^N\times\mathbb{R}^N}&\frac{|u(x)-u(y)|^p}{|x-y|^{N+s\,p}} \, dx \, dy \\
&= \frac{1}{2} \int_{\mathbb{S}^{N-1}} \int_{\{h\in\mathbb{R}^N:\langle h, \omega \rangle = 0\}}\left( \iint_{\mathbb{R}\times\mathbb{R}} \frac{|u(h+t\, \omega)-u(h + \tau\, \omega)|^p}{|t-\tau|^{1+s\,p}} \, dt \, d\tau\right) \, dh \, d\mathcal{H}^{N-1}(\omega),
\end{split}
\]
where $dh$ denotes the $(N-1)-$dimensional Lebesgue measure on the hyperplane $\{h\in\mathbb{R}^N:\langle h, \omega \rangle = 0\}$.
Such a formula can be proved exactly as \cite[Lemma 2.4]{LS}. By starting from this, it is sufficient to use the one-dimensional Hardy inequality of Proposition \ref{prop:dyda} for the integral on $\mathbb{R}\times\mathbb{R}$, in place of the inequality of \cite[Theorem 2.1]{LS}. 
We leave the details to the reader.
\end{proof}

\appendix

\section{Negative powers of the distance in the borderline case $s=1/2$}
\label{app:B}
We consider for $-1<\beta<0$
\[
U_\beta(t)=d_I(t)^\beta=\left(\min\{t,1-t\}\right)^\beta,\qquad \mbox{ for }t\in I=(0,1).
\]
We extend this function by $0$ outside the interval $I$. We want to estimate its fractional Laplacian of order $1/2$.
\begin{lm}
\label{lm:bum!}
Under the assumptions above, we have 
\begin{equation}
\label{claim}
(-\Delta)^\frac{1}{2} U_\beta\le \beta\,H(t)\,U_\beta(t)+\frac{2}{t\,(1-t)}\,U_\beta(t),\qquad \mbox{ in } I,
\end{equation}
in weak sense,
where $H$ is the continuous function on $I\setminus\{1/2\}$ defined by
\[
H(t)=-\frac{2}{t\,(1-t)}+\frac{2}{d_I(t)}\,\log\left(\frac{4\,t\,(1-t)}{(1-2\,t)^2}\right).
\] 
This has the following properties:
\begin{itemize}
\item it is symmetric with respect to $1/2$, i.e.
\[
H(t)=H(1-t),\qquad \mbox{ for } t\in I\setminus\left\{\frac{1}{2}\right\};
\]
\item it belongs to $L^q_{\rm loc}(I)$ for every $1\le q<\infty$;
\vskip.2cm
\item it satisfies
\[
H(t)\sim - 4\,\log\left(\frac{1}{2}-t\right)^2,\qquad \mbox{ for } t\to \frac{1}{2},
\]
thus the right-hand side of \eqref{claim} diverges to $-\infty$ as $t$ approaches $1/2$.
\end{itemize}
In particular, in this case the function $U_\beta$ is not even locally weakly $(1/2)-$superharmonic on $I$.
\end{lm} 
\begin{proof}
We first show that $U_\beta$ satisfies \eqref{claim} in $I\setminus\{1/2\}$.
Let us take $\varphi\in C^\infty_0(I\setminus\{1/2\})$ non-negative, then there exists $0<\delta_0<1/4$ such that its support is contained in the set
\[
\mathcal{I}_{\delta_0}=\left[\delta_0,\frac{1}{2}-\delta_0\right]\cup\left[\frac{1}{2}+\delta_0,1-\delta_0\right].
\] 
For every $\varepsilon>0$ and $t\in I$, we set 
\[
\mathfrak{I}_\varepsilon(t)=(t-\varepsilon,t+\varepsilon)\qquad \mbox{ and }\qquad \mathcal{D}_\varepsilon=\Big\{(t,y)\in\mathbb{R}\times\mathbb{R}\, :\, t-\varepsilon\le y\le t+\varepsilon\Big\}.
\]
By using Fubini's Theorem and a change of variable, we can write as usual
\begin{equation}
\label{parte}
\begin{split}
\iint_{\mathbb{R}\times\mathbb{R}}& \frac{\Big(U_\beta(t)-U_\beta(y)\Big)\,\big(\varphi(t)-\varphi(y)\big)}{|t-y|^{2}}\,dt\,dy\\
&=\lim_{\varepsilon\to 0} \iint_{(\mathbb{R}\times\mathbb{R})\setminus\mathcal{D}_\varepsilon}\frac{\Big(U_\beta(t)-U_\beta(y)\Big)\,\big(\varphi(t)-\varphi(y)\big)}{|t-y|^{2}}\,dt\,dy\\
&=2\,\lim_{\varepsilon\to 0} \int_{\mathbb{R}} \left(\int_{\mathbb{R}\setminus \mathfrak{I}_\varepsilon(t)} \frac{U_\beta(t)-U_\beta(y)}{|t-y|^{2}}\,dy\right)\,\varphi(t)\,dt\\
&=2\,\lim_{\varepsilon\to 0} \int_{\mathcal{I}_{\delta_0}} \left(\int_{\mathbb{R}\setminus \mathfrak{I}_\varepsilon(t)} \frac{U_\beta(t)-U_\beta(y)}{|t-y|^{2}}\,dy\right)\,\varphi(t)\,dt.
\end{split}
\end{equation}
We first observe that, by using that $U_\beta(t)=U_\beta\left(1-t\right)$, 
we have for $t\in \mathcal{I}_{\delta_0}$
\[
\begin{split}
\int_{\mathbb{R}\setminus \mathfrak{I}_\varepsilon(t)} \frac{U_\beta(t)-U_\beta(y)}{|t-y|^{2}}\,dy
&=\int_{\mathbb{R}\setminus \mathfrak{I}_\varepsilon(t)} \frac{U_\beta(1-t)-U_\beta(1-y)}{|t-y|^{2}}\,dy\\
&=\int_{\mathbb{R}\setminus \mathfrak{I}_\varepsilon(t)}\frac{U_\beta(1-t)-U_\beta(1-y)}{|(1-t)-(1-y)|^{2}}\,dy=\int_{\mathbb{R}\setminus \mathfrak{I}_\varepsilon(1-t)} \frac{U_\beta(1-t)-U_\beta(\tau)}{|(1-t)-\tau|^{2}}\,d\tau.
\end{split}
\]
This shows that it is sufficient to consider $t\in[\delta_0,1/2-\delta_0]$.
For almost every $t\in[\delta_0,1/2-\delta_0]$ and every $0<\varepsilon< \delta_0$,
we have\footnote{Observe that by construction $t-\varepsilon> 0$ and $t+\varepsilon<1/2$.}
\[
\begin{split}
\int_{\mathbb{R}\setminus \mathfrak{I}_\varepsilon(t)} \frac{U_\beta(t)-U_\beta(y)}{|t-y|^{2}}\,dy&=\int_0^{t-\varepsilon}\frac{t^\beta-y^\beta}{|t-y|^{2}}\,dy+\int_{t+\varepsilon}^\frac{1}{2} \frac{t^\beta-y^\beta}{|t-y|^{2}}\,dy+\int_\frac{1}{2}^1 \frac{t^\beta-(1-y)^\beta}{(y-t)^{2}}\,dy\\
&+\int_1^{+\infty} \frac{t^\beta}{(y-t)^{2}}\,dy+\int_{-\infty}^0 \frac{t^\beta}{(t-y)^{2}}\,dy\\
&=\int_0^{t-\varepsilon}\frac{t^\beta-y^\beta}{|t-y|^{2}}\,dy+\int_{t+\varepsilon}^\frac{1}{2} \frac{t^\beta-y^\beta}{|t-y|^{2}}\,dy+\int_\frac{1}{2}^1 \frac{t^\beta-(1-y)^\beta}{(y-t)^{2}}\,dy\\
&+\frac{U_\beta(t)}{(1-t)}+\frac{U_\beta(t)}{t}.
\end{split}
\] 
To estimate the remaining integrals, we use the ``above tangent'' property for the convex function $\tau\mapsto\tau^\beta$, to infer that
\[
t^\beta-y^\beta\le \beta\,t^{\beta-1}\,(t-y),\qquad \mbox{ for } y\in\left(0,t-\varepsilon\right)\cup\left(t+\varepsilon,\frac{1}{2}\right),
\]
and 
\[
t^\beta-(1-y)^\beta\le \beta\,t^{\beta-1}\,(t+y-1),\qquad \mbox{ for } y\in\left(\frac{1}{2},1\right).
\]
These yield
\[
\begin{split}
\int_0^{t-\varepsilon} \frac{t^\beta-y^\beta}{|t-y|^{2}}\,dy&+\int_{t+\varepsilon}^\frac{1}{2} \frac{t^\beta-y^\beta}{|t-y|^{2}}\,dy+\int_\frac{1}{2}^1 \frac{t^\beta-(1-y)^\beta}{(y-t)^{2}}\,dy\\
&\le \beta\,t^{\beta-1}\,\int_0^{t-\varepsilon} \frac{t-y}{|t-y|^{2}}\,dy+\beta\,t^{\beta-1}\,\int_{t+\varepsilon}^\frac{1}{2} \frac{t-y}{|t-y|^{2}}\,dy+\beta\,t^{\beta-1}\,\int_\frac{1}{2}^1 \frac{(t+y-1)}{(y-t)^{2}}\,dy.
\end{split}
\]
The last integrals can be explicitly computed. We have 
\[
\begin{split}
\int_0^{t-\varepsilon} \frac{t-y}{|t-y|^{2}}\,dy&+\int_{t+\varepsilon}^{1/2} \frac{t-y}{|t-y|^{2}}\,dy
=\log t-\log\left(\frac{1}{2}-t\right),
\end{split}
\]
and
\[
\begin{split}
\int_{1/2}^1 \frac{(t+y-1)}{(y-t)^2}\,dy&=t\,\int_{1/2}^1 \frac{1}{(y-t)^2}\,dy+\int_{1/2}^1 \frac{y-1}{(y-t)^{2}}\,dy\\
&=t\,\left[\left(\frac{1}{2}-t\right)^{-1}-(1-t)^{-1}\right]\\
&-\frac{1}{2}\,\left(\frac{1}{2}-t\right)^{-1}+\left[\log(1-t)-\log\left(\frac{1}{2}-t\right)\right].
\end{split}
\]
This finally gives that for $t\in[\delta_0,1/2-\delta_0]$, we have
\begin{equation}\label{eq:minore}
\int_{\mathbb{R}\setminus \mathfrak{I}_\varepsilon(t)} \frac{U_\beta(t)-U_\beta(y)}{|t-y|^{1+2\,s}}\,dy\le \beta\,t^{\beta-1}\,G(t)+\frac{U_\beta(t)}{t\,(1-t)},
\end{equation}
where we set for simplicity
\[
\begin{split}
G(t)&=\left[\log t-\log\left(\frac{1}{2}-t\right)\right]+t\,\left[\left(\frac{1}{2}-t\right)^{-1}-(1-t)^{-1}\right]\\
&-\frac{1}{2}\,\left(\frac{1}{2}-t\right)^{-1}+\left[\log(1-t)-\log\left(\frac{1}{2}-t\right)\right].
\end{split}
\]
With simple manipulations, we see that this can be also written as
\[
\begin{split}
G(t)
&=-\frac{1}{1-t}+\log \left(\frac{4\,t\,(1-t)}{\left(1-2\,t\right)^2}\right),
\end{split}
\]
and thus this is a continuous function on $(0,1/2)$ such that 
\[
G\in L^q_{\rm loc}((0,1/2]),\ \mbox{ for every } 1\le q<\infty\qquad \mbox{ and }\qquad \lim_{t\to \left(\frac{1}{2}\right)^-} G(t)=+\infty,
\] 
because of the logarithmic term.
If we now define
\[
H(t):=2\,\frac{G(t)}{t},\quad \mbox{ for } t\in \left(0,\frac{1}{2}\right)\qquad \mbox{ and }\qquad H(t):=H(1-t),\quad \mbox{ for } t\in \left(\frac{1}{2},1\right),
\]
from \eqref{parte} and \eqref{eq:minore}, by recalling that $\varphi$ is non-negative we finally get
\[
\begin{split}
\iint_{\mathbb{R}\times\mathbb{R}}& \frac{\Big(U_\beta(t)-U_\beta(y)\Big)\,\big(\varphi(t)-\varphi(y)\big)}{|t-y|^{2}}\,dt\,dy\le \int_I \left[\beta\,H(t)+ \frac{2}{t\,(1-t)}\right]\,U_\beta(t)\,\varphi(t)\,dt.
\end{split}
\]
This shows that $U_\beta$ is a local weak subsolution of the the claimed equation \eqref{claim}, at least in the open set $I\setminus\{1/2\}$.
\vskip.2cm\noindent
In order to show that $U_\beta$ is a local weak subsolution on the whole interval $I$, it is sufficient to use a standard trick to ``fill the hole'': we take $\varphi\in C^\infty_0(I)$ non-negative and for every natural number $n\ge 5$ we take $\psi_n\in C^\infty(\overline{I})$ such that 
\[
\psi_n\equiv 1 \mbox{ on } \overline{I}\setminus\left[\frac{1}{2}-\frac{2}{n},\frac{1}{2}+\frac{2}{n}\right],\qquad \psi_n\equiv 0 \mbox{ on }\left[\frac{1}{2}-\frac{1}{n},\frac{1}{2}+\frac{1}{n}\right],
\]
and 
\[
0\le \psi_n\le 1,\qquad |\psi_n'|\le C\,n.
\]
The seminorm of $\psi_n$ can be estimated by using its properties and an interpolation inequality (see \cite[Corollary 2.2]{BPS}), i.\,e.
\[
\begin{split}
[\psi_n]^2_{W^{\frac{1}{2},2}(I)}=[1-\psi_n]_{W^{\frac{1}{2},2}(I)}^2&\le C\,\left(\int_I |\psi_n'|^2\,dt\right)^\frac{1}{2}\,\left(\int_I |1-\psi_n|^2\,dt\right)^\frac{1}{2}\\
&\le C\,\left(\int_{\frac{1}{2}-\frac{2}{n}}^{\frac{1}{2}+\frac{2}{n}}|\psi_n'|^2\,dt\right)^\frac{1}{2}\,\left(\int_{\frac{1}{2}-\frac{2}{n}}^{\frac{1}{2}+\frac{2}{n}} |1-\psi_n|^2\,dt\right)^\frac{1}{2}\le \widetilde{C},
\end{split}
\]
where $\widetilde{C}$ does not depend on $n$. This in particular implies that the sequence $\{\Psi_n\}_{n\ge 5}\subset L^2(I\times I)$ defined by
\begin{equation}
\label{bettino}
\Psi_n(t,y):=\frac{\psi_n(t)-\psi_n(y)}{|t-y|},\qquad \mbox{ for a.\,e. } (t,y)\in I\times I,
\end{equation}
is bounded in $L^2(I\times I)$, since by construction
\[
\|\Psi_n\|_{L^2(I\times I)}=[\psi_n]_{W^{\frac{1}{2},s}(I)}.
\] 
Thus, up to a subsequence, it converges weakly in $L^2(I\times I)$. Thanks to the properties of $\psi_n$, such a limit function must coincide with the null one.
\par
The test function $\varphi\,\psi_n$ belongs to $C^\infty_0(I\setminus\{1/2\})$ and is non-negative. From the first part we get
\begin{equation}
\label{hole!}
\begin{split}
\iint_{\mathbb{R}\times\mathbb{R}}& \frac{\Big(U_\beta(t)-U_\beta(y)\Big)\,\big(\varphi(t)\,\psi_n(t)-\varphi(y)\,\psi_n(y)\big)}{|t-y|^{2}}\,dt\,dy\\
&\le \int_I \left[H(t)\,U_\beta(t)+ \frac{2}{t\,(1-t)}\right]\,U_\beta(t)\,\varphi(t)\,\psi_n(t)\,dt.
\end{split}
\end{equation}
We wish to pass to the limit in \eqref{hole!}, as $n$ goes to $\infty$: for the right-hand side, it is easily seen that 
\[
\begin{split}
\lim_{n\to\infty}\int_I \left[H(t)\,U_\beta(t)+ \frac{2}{t\,(1-t)}\right]&\,U_\beta(t)\,\varphi(t)\,\psi_n(t)\,dt=\int_I \left[H(t)\,U_\beta(t)+ \frac{2}{t\,(1-t)}\right]\,U_\beta(t)\,\varphi(t)\,dt,
\end{split}
\]
by the Dominated Convergence Theorem. As for the left-hand side, we split the integral as follows:
\[
\begin{split}
\iint_{\mathbb{R}\times\mathbb{R}}& \frac{\Big(U_\beta(t)-U_\beta(y)\Big)\,\big(\varphi(t)\,\psi_n(t)-\varphi(y)\,\psi_n(y)\big)}{|t-y|^{2}}\,dt\,dy\\
&=\iint_{I''\times I''}\frac{\Big(U_\beta(t)-U_\beta(y)\Big)\,\big(\varphi(t)\,\psi_n(t)-\varphi(y)\,\psi_n(y)\big)}{|t-y|^{2}}\,dt\,dy\\
&+2\,\iint_{I'\times (\mathbb{R}\setminus I'')}\frac{\Big(U_\beta(t)-U_\beta(y)\Big)\,\varphi(t)\,\psi_n(t)}{|t-y|^{2}}\,dt\,dy,
\end{split}
\]
where $I'\Subset I''\Subset I$ and $I'$ contains the support of $\varphi$. For the last integral we can easily pass to the limit as $n$ goes to $\infty$, for the first one we proceed as follows
\[
\begin{split}
\iint_{I''\times I''}&\frac{\Big(U_\beta(t)-U_\beta(y)\Big)\,\big(\varphi(t)\,\psi_n(t)-\varphi(y)\,\psi_n(y)\big)}{|t-y|^{2}}\,dt\,dy\\
&=\iint_{I''\times I''}\frac{\Big(U_\beta(t)-U_\beta(y)\Big)\,\big(\varphi(t)-\varphi(y)\big)}{|t-y|^{2}}\,\frac{\psi_n(t)+\,\psi_n(y)}{2}dt\,dy\\
&+\iint_{I''\times I''}\frac{\Big(U_\beta(t)-U_\beta(y)\Big)\,\big(\psi_n(t)-\psi_n(y)\big)}{|t-y|^{2}}\,\frac{\varphi(t)+\,\varphi(y)}{2}dt\,dy.
\end{split}
\]
By using that 
\[
\frac{\Big(U_\beta(t)-U_\beta(y)\Big)\,\big(\varphi(t)-\varphi(y)\big)}{|t-y|^2}\in L^1(I''\times I''),
\]
and the properties of $\psi_n$, it is easily seen that  
\[
\begin{split}
\lim_{n\to\infty}\iint_{I''\times I''}&\frac{\Big(U_\beta(t)-U_\beta(y)\Big)\,\big(\varphi(t)-\varphi(y)\big)}{|t-y|^{2}}\,\frac{\psi_n(t)+\,\psi_n(y)}{2}dt\,dy\\
&=\iint_{I''\times I''}\frac{\Big(U_\beta(t)-U_\beta(y)\Big)\,\big(\varphi(t)-\varphi(y)\big)}{|t-y|^{2}}\,dt\,dy,
\end{split}
\]
again thanks to the Dominated Convergence Theorem. Finally, the last integral is the most delicate one: with the notation \eqref{bettino}, we can write 
\[
\frac{\Big(U_\beta(t)-U_\beta(y)\Big)\,\big(\psi_n(t)-\psi_n(y)\big)}{|t-y|^{2}}\,\frac{\varphi(t)+\,\varphi(y)}{2}=\Phi(t,y)\,\Psi_n(t,y),
\]
where 
\[
\Phi(t,y)=\frac{\Big(U_\beta(t)-U_\beta(y)\Big)}{|t-y|}\,\frac{\varphi(t)+\,\varphi(y)}{2}\in L^2(I''\times I'').
\]
Thus, by using the weak convergence of $\{\Psi_n\}_{n\ge 5}$ previously discussed, we get
\[
\begin{split}
\lim_{n\to\infty}\iint_{I''\times I''}\frac{\Big(U_\beta(t)-U_\beta(y)\Big)\,\big(\psi_n(t)-\psi_n(y)\big)}{|t-y|^{2}}\,\frac{\varphi(t)+\,\varphi(y)}{2}dt\,dy=0.
\end{split}
\]
Finally, we obtain that we can pass to the limit in \eqref{hole!} as $n$ goes to $\infty$ and obtain
\[
\iint_{\mathbb{R}\times\mathbb{R}} \frac{\Big(U_\beta(t)-U_\beta(y)\Big)\,\big(\varphi(t)-\varphi(y)\big)}{|t-y|^{2}}\,dt\,dy\\
\le \int_I \left[H(t)\,U_\beta(t)+2\, \frac{U_\beta(t)}{(1-t)}+2\, \frac{U_\beta(t)}{t}\right]\,\varphi(t)\,dt,
\]
for every $\varphi\in C^\infty_0(I)$ non-negative, as desired.
\end{proof}

\end{document}